\documentclass[reqno, 11pt, a4paper]{amsart}
\usepackage{amsfonts, amsmath, amssymb, amsthm, color}
\usepackage[hmargin=2.75cm, vmargin=2.9cm]{geometry}
\usepackage{nccmath}

\newtheorem{thm}{Theorem}[section]
\newtheorem{lemma}[thm]{Lemma}
\newtheorem{prop}[thm]{Proposition}
\newtheorem{cor}[thm]{Corollary}

\theoremstyle{definition}
\newtheorem{rmk}[thm]{Remark}

\newcommand{\ga}{\gamma}
\newcommand{\ep}{\epsilon}
\newcommand{\vep}{\varepsilon}
\newcommand{\vph}{\varphi}
\newcommand{\pa}{\partial}

\newcommand{\N}{\mathbb{N}}
\newcommand{\R}{\mathbb{R}}
\renewcommand{\S}{\mathbb{S}}

\newcommand{\mbd}{\mathbf{d}}
\newcommand{\mbP}{\mathbf{P}}
\newcommand{\mbxi}{\boldsymbol \xi}

\newcommand{\mce}{\mathcal{E}}

\newcommand{\mci}{\mathcal{I}}
\newcommand{\mck}{\mathcal{K}}

\newcommand{\mcn}{\mathcal{N}}

\newcommand{\mfh}{\mathfrak{H}}
\newcommand{\mfj}{\mathfrak{J}}

\newcommand{\whh}{\widehat{H}}

\newcommand{\tga}{\tilde{\gamma}}
\newcommand{\tta}{\tilde{\tau}}
\newcommand{\wtf}{\widetilde{F}}
\newcommand{\wtg}{\widetilde{G}}
\newcommand{\wth}{\widetilde{H}}
\newcommand{\wtq}{\widetilde{Q}}
\newcommand{\wtph}{\widetilde{\Phi}}
\newcommand{\wtps}{\widetilde{\Psi}}

\newcommand{\ovom}{\overline{\Omega}}

\newcommand{\dxi}{{\mbd, \mbxi}}
\newcommand{\dxin}{{\mbd_n, \mbxi_n}}
\newcommand{\dxii}{{\mbd_{\infty}, \mbxi_{\infty}}}

\renewcommand{\(}{\left(}
\renewcommand{\)}{\right)}

\allowdisplaybreaks
\numberwithin{equation}{section}

\begin{document}
\title[Blowing-up solutions to critical elliptic systems in bounded domains]{Multiple blowing-up solutions to \\ critical elliptic systems in bounded domains}

\author{Seunghyeok Kim}
\address[Seunghyeok Kim]{Department of Mathematics and Research Institute for Natural Sciences, College of Natural Sciences, Hanyang University, 222 Wangsimni-ro Seongdong-gu, Seoul 04763, Republic of Korea}
\email{shkim0401@hanyang.ac.kr shkim0401@gmail.com}

\author{Angela Pistoia}
\address[A.~Pistoia]{Dipartimento SBAI,  ``Sapienza" Universit\`a di Roma, via Antonio Scarpa 16, 00161 Roma, Italy}
\email{angela.pistoia@uniroma1.it}

\begin{abstract}
We construct families of blowing-up solutions to elliptic systems on smooth bounded domains in the Euclidean space,
which are variants of the critical Lane-Emden system and analogous to the Brezis-Nirenberg problem.
We find a function which governs blowing-up points and rates, observing that it reflects the strong nonlinear characteristic of the system.
By using it, we also prove that a single blowing-up solution exists in general domains,
and construct examples of contractible domains where multiple blowing-up solutions are allowed to exist.
We believe that a variety of new ideas and arguments developed here will help to analyze blowing-up phenomena in related Hamiltonian-type systems.
\end{abstract}

\date{\today}
\subjclass[2010]{Primary: 35J47, Secondary: 35B33, 35B40, 35B44}
\keywords{Lane-Emden systems, Brezis-Nirenberg problem, critical hyperbola, multiple blowing-up solution, asymptotic profile}
\thanks{S. Kim was supported by Basic Science Research Program through the National Research Foundation of Korea (NRF) funded by the Ministry of Education (NRF2017R1C1B5076384, NRF2020R1C1C1A01010133)
and associate member program of Korea Institute for Advanced Study (KIAS).
A. Pistoia was partially supported by Fondi di Ateneo ``Sapienza" Universit\`a di Roma (Italy) 
and project Vain-Hopes within the program VALERE: VAnviteLli pEr la RicErca.}
\maketitle

\section{Introduction}
In this paper, we construct families of solutions to an elliptic system
\begin{equation}\label{eq:LEs}
\begin{cases}
-\Delta u = |v|^{p-1}v + \ep (\alpha u + \beta_1 v) &\text{in } \Omega,\\
-\Delta v = |u|^{q-1}u + \ep (\beta_2 u + \alpha v) &\text{in } \Omega,\\
u,v= 0 &\text{on } \pa \Omega
\end{cases}
\end{equation}
where $\Omega$ is a smooth bounded domain in $\R^N,$ $N \ge 3$, $\ep > 0$ is a small parameter,
$\alpha,$ $\beta_1$ and $\beta_2$ are real numbers, and $(p,q)$ is a pair of positive numbers lying on the {\em critical hyperbola}
\begin{equation}\label{eq:hyper}
\frac{1}{p+1} + \frac{1}{q+1} = \frac{N-2}{N}.
\end{equation}
Without loss of generality, we may assume that $p \le \frac{N+2}{N-2} \le q$.

System \eqref{eq:LEs} is analogous to Brezis-Nirenberg problem \cite{bn}
\begin{equation}\label{bn}
\begin{cases}
-\Delta w = |w|^{4\over N-2}w + \ep \beta w &\hbox{in}\ \Omega,\\
w = 0 &\hbox{on}\ \pa\Omega.
\end{cases}
\end{equation}
It is well-known that the classical Pohozaev's identity \cite{p}
implies that if $\ep \beta \le 0$ and $\Omega$ is star-shaped, then \eqref{bn} has no solution.
On the other hand, \eqref{bn} always has a solution if $N \ge 4$ and $\ep \beta > 0$, which is positive for $0<\ep\beta < \lambda_1(\Omega)$ (see \cite{CFP}).
Here and after, $\lambda_n(\Omega)$ is the $n$-th eigenvalue of the Laplacian $-\Delta$ with Dirichlet boundary condition in $\Omega.$

For system \eqref{eq:LEs}, Mitidieri \cite{m1} and Van der Vorst \cite{v} derived a Pohozaev-type identity (see Theorem 1 in \cite{hmv}),
which implies that there exists no positive solution in a star-shaped domain $\Omega$ provided that the matrix
\[\(\begin{matrix}
-{\beta_2(q-1)\over 2(q+1)} & -\frac\alpha N\\
-\frac\alpha N&-{\beta_1(p-1)\over 2(p+1)} \\
\end{matrix}\)\]
is positive semi-definite. In particular, we have the non-existence result if $p,q>1,$ $\alpha=0$ and $\beta_1,\beta_2 \le 0$.
On the other hand, Hulshof, Mitidieri and Van der Vorst \cite{hmv} proved that if $p,q>1,$ $\alpha\ge 0$ and either $\beta_1 > 0$  or $\beta_1 = 0$ and $\beta_2 > 0$,
then \eqref{eq:LEs} has a solution provided that $\ep^2\beta_1\beta_2 \ne \lambda^2_n(\Omega)$ for all $n \in \N$ and $N$ is sufficiently large.
They also showed that the solution is positive if $\beta_1, \beta_2 > 0$ and $\ep^2\beta_1\beta_2 < \lambda^2_1(\Omega)$.
Their approach relies on a dual formulation due to Clarke and Ekeland \cite{ce}.
If $\beta_1 = \alpha = 0$ and $\beta_2 > 0$, system \eqref{eq:LEs} can be reduced to a higher-order single equation as in \eqref{eq:LEs2}.
Based on this fact, Guerra \cite{Ge} studied asymptotic behavior of ground state solutions as $\ep \to 0$.
For unbounded domains $\Omega$, Colin and Frigon \cite{cf} established the existence of solutions to \eqref{eq:LEs} under the assumptions $\alpha=0$ and $0 < \ep\beta_1,\ep\beta_2 < \lambda_1(\Omega).$

\medskip
It is worthwhile to point out that an important feature of problem \eqref{bn} is the existence of positive and sign-changing solutions $w_\ep$ to \eqref{bn}
which blow-up at one or more points in $\Omega$ as $\ep\to0$ (see e.g., Musso and Pistoia \cite{MP},  Rey \cite{R} and Bartsch, Micheletti and Pistoia \cite{BMP}).
The shape of these solutions close to each blow-up point looks like a positive or a negative bubble, namely,
\begin{equation}\label{eq:bubW}w_\ep(x)\simeq \pm W_{\mu,\xi}(x):=\pm\mu^{-{N-2\over2}}W \(x-\xi\over\mu\)\ \hbox{with}\ \mu=\mu(\ep)\to0\ \hbox{as}\ \ep\to0\end{equation}
where the standard bubble $W=W_{1,0}$ is the unique positive ground state solution in $\dot{H}^1(\R^N)$ to
\[-\Delta W = W^p \quad{\text{in }} \R^N\]
such that $W(0) = \max_{x \in \R} W(x) = 1.$

Therefore, it is natural to ask if a similar phenomena also happens for system \eqref{eq:LEs}. In particular, we address the following question.
\begin{itemize}
\item[$Q$.] {\em Does there exist positive or sign-changing solutions to  \eqref{eq:LEs} which blow-up or blow-down at one or more points in $\Omega$ as $\ep \to 0$?}
\end{itemize}
We will give some answers.

\medskip
Our first theorem is the existence of a one-point blowing-up solution in general domains.
\begin{thm}\label{thm:main1}
Assume that $N \ge 8$, $p \in (1, \frac{N-1}{N-2})$, and $(p,q)$ satisfies \eqref{eq:hyper}.
Then there exists a small number $\ep_0 > 0$ depending only on $N$, $p$, $\Omega$, $\alpha$, $\beta_1$ and $\beta_2$
such that for any $\ep\in(0,\ep_0)$, system \eqref{eq:LEs} has a solution in $(C^2(\ovom))^2$ which blows-up at one point in $\Omega$ as $\ep \to 0$
provided that one of the following conditions is satisfied:
\begin{equation}\tag{C}\label{C}
\begin{aligned}
(i)& \quad \beta_1>0, \\
(ii)& \quad \beta_1=0 \text{ and } \alpha>0, \\
(iii)& \quad \beta_1=\alpha=0 \text{ and } \beta_2>0.
\end{aligned}
\end{equation}

Moreover, if $\beta_1, \beta_2 \ge 0$, then \eqref{eq:LEs} has a solution with positive components showing the prescribed blowing-up behavior.
\end{thm}
\noindent A more precise asymptotic profile of the solutions, including their blow-up rates and locations, is found in Subsection \ref{subsec:proof11} and Proposition \ref{prop:red}.

\medskip
The next result exhibits an example of contractible domains with {\em rich} geometry,
where problem \eqref{eq:LEs} has solutions with one or more positive blow-up points.
Let us introduce the {\em dumbbell-shaped} domain $\Omega_{\eta}$
that we obtain by connecting $l$ disjoint domains $\Omega^*_1, \dots, \Omega^*_l$ with $l-1$ necks of thickness less than a small number $\eta > 0$.
Its precise description is given as follows. We assume that given numbers $a_1 < b_1 < a_2 < \cdots < b_{l-1} < a_l < b_l$,
\[\Omega^*_i \subset \left\{ (x_1,x')\in\R\times\R^{N-1}: a_i \le x_1 \le b_i \right\}
\quad \text{and} \quad
\Omega^*_i \cap \left\{ (x_1,x')\in\R\times\R^{N-1}: x'=0 \right\} \ne \emptyset\]
for $i = 1, \cdots, l$, and set the $\eta$-neck
\[\mcn_\eta = \left\{(x_1,x') \in \R\times\R^{N-1} : x_1 \in (a_1, b_k),\, |x'| < \eta\right\}.\]
Let $\Omega_0 = \cup_{i=1}^l \Omega^*_i$, and $\{\Omega_{\eta}\}_{\eta > 0}$ be a family of smooth (connected) domains such that
\[\Omega_0 \subset \Omega_{\eta} \subset  \Omega_0 \cup \mcn_\eta
\quad \text{and} \quad \Omega_{\eta_1} \subset \Omega_{\eta_2} \quad \text{for } \eta_1 \le \eta_2.\]
\begin{thm}\label{thm:main2}
Assume that $N \ge 8$, $p \in (1, \frac{N-1}{N-2})$, $(p,q)$ satisfies \eqref{eq:hyper} and $k \in \{1, \cdots, l\}$.
Then there exist small numbers $\eta_0,\, \ep_0 > 0$ such that for any $\ep \in (0,\ep_0)$ and $\eta \in (0,\eta_0)$,
system \eqref{eq:LEs} with $\Omega = \Omega_{\eta}$ has ${l \choose k}$ solutions in $(C^2(\ovom))^2$ which blow-up at $k$ points as $\ep \to 0$
provided that one of the conditions (i), (ii) and (iii) in \eqref{C} is satisfied.

Moreover, if $\beta_1, \beta_2 \ge 0$, then \eqref{eq:LEs} has ${l \choose k}$ solutions with positive components showing the prescribed blowing-up behavior.
\end{thm}
\noindent Again, a more precise asymptotic profile of the solutions is found in Subsection \ref{subsec:proof12} and Proposition \ref{prop:red}.
The function $\mfj^{\eta}$ in \eqref{eq:Jep2} (see also \eqref{eq:fmu}, \eqref{eq:wth} and \eqref{eq:wtg})
that governs the blow-up rates and locations manifests the strong nonlinear characteristic of system \eqref{eq:LEs}.

\medskip
As far as it concerns the existence of solutions with sign-changing blow-up points,
in Section \ref{sec:nod} we give an abstract result which allows us to find them via the existence of critical points of
the function \eqref{nodal}, which rules the locations and the rates of the concentration points.
However, it turns out to be extremely difficult to find these critical points, even in the simplest case of one positive and one negative blow-up points.
Eventually, an interesting question naturally arises: under the assumptions of Theorem \ref{thm:main1},
{\it does problem \eqref{eq:LEs} have a sign-changing solution with one positive and one negative blow-up points?}
The answer is positive if condition \eqref{conj} holds true, so it would be worth proving or disproving it.

Finally, in  Section \ref{sec:subc} we apply the strategy developed in the previous sections   to build positive  solutions to slightly subcritical systems which blow-up at one or more points in $\Omega,$ in the same spirit of what Bahri, Li and Rey proved for the single equation (see \cite{BLR}).

\medskip
The proof of our results depends on the finite-dimensional Ljapunov-Schmidt reduction method.
We will build solutions $(u_\ep,v_\ep)$ to \eqref{eq:LEs} which blows-up at some points in $\Omega$ whose shape around each blow-up point $\xi \in \Omega$ resembles a {\em bubble}, i.e.
\begin{multline}\label{eq:A}
(u_\ep(x),v_\ep(x)) \simeq \(U_{\mu,\xi}(x),V_{\mu,\xi}(x)\):=\(\mu^{-{N \over q+1}} U \(x-\xi\over\mu\),\, \mu^{-{N \over p+1}} V\(x-\xi\over\mu\)\) \\
\hbox{with}\ \mu=\mu(\ep)\to0\ \hbox{as}\ \ep\to0
\end{multline}
where the standard bubble $(U,V)=(U_{1,0},V_{1,0})$ is the unique positive ground state solution in $\dot{W}^{2,{p+1 \over p}}(\R^N) \times \dot{W}^{2,{q+1 \over q}}(\R^N)$ to
\begin{equation}\label{eq:bubble2}
\begin{cases}
-\Delta U = V^p \quad \text{in } \R^N,\\
-\Delta V = U^q \quad \text{in } \R^N
\end{cases}
\end{equation}
such that $U(0) = \max_{x \in \R} U(x) = 1$ (see Subsection \ref{subsec:bubble}).

Although the proof follows the standard steps of the reduction procedure,
it is not trivial at all and requires a lot of works and new ideas on each step.

A first key point in the construction of the solutions is the non-degeneracy of the bubbles, which has been recently proved by Frank, Kim and Pistoia \cite{FKP}.

A second key point is finding a good ansatz and this is the main issue of the proof.
The bubble in \eqref{eq:A} is a good approximation close to the blow-up point but far away it does not fit the Dirichlet boundary condition.
In the case of the single equation \eqref{bn}, one can overcome this issue by introducing the {\em projection} operator $P: H^1(\R^N) \to H^1_0(\Omega)$ defined by
\[\begin{cases}
\Delta PW_{\mu,\xi} =\Delta W_{\mu,\xi}  &\text{in } \Omega,\\
PW _{\mu,\xi}=0 &\text{on } \pa\Omega
\end{cases}\]
(see for example \cite{MP,R}), where  $W_{\mu,\xi}$ is the bubble introduced in \eqref{eq:bubW}. The function $PW_{\mu,\xi}$ is a perfect ansatz to build solution to the single equation \eqref{bn}.

Unfortunately, this {\em linear} projection is not satisfactory for system \eqref{eq:LEs} at least for $p < {N \over N-2}$.
To understand it, let us reduce system \eqref{eq:LEs} with $\ep=0$ to a scalar equation via an {\em inversion} argument,
which heuristically consists of taking $V_{\mu,\xi} = -|\Delta U_{\mu,\xi}|^{\frac{1-p}{p}}\Delta U_{\mu,\xi}$ and plugging it into the first equation of \eqref{eq:LEs}.
Here, $\(U_{\mu,\xi},V_{\mu,\xi}\)$ is the bubble introduced in \eqref{eq:A}.
Then we find the single equation of higher order
\begin{equation}\label{eq:LEs2}
\begin{cases}
\Delta\(|\Delta U_{\mu,\xi} |^{\frac{1-p}{p}}\Delta U_{\mu,\xi}\) = |U_{\mu,\xi}|^{q-1}U_{\mu,\xi} &\text{in } \Omega,\\
U _{\mu,\xi}= \Delta U _{\mu,\xi}= 0 &\text{on } \pa\Omega.
\end{cases}
\end{equation}
Now, this suggests that a good approximation is the {\em nonlinear} projection $P_{p,q} U_{\mu,\xi}$ of the bubble $U_{\mu,\xi}$,
that is, the solution of the {\em nonlinear} Dirichlet problem
\[\begin{cases}
\Delta\(|\Delta P_{p,q}U_{\mu,\xi}|^{\frac{1-p}{p}} \Delta P_{p,q}U_{\mu,\xi}\) = |U_{\mu,\xi}|^{q-1}U_{\mu,\xi} &\text{in } \Omega,\\
P_{p,q}U_{\mu,\xi} = \Delta P_{p,q}U_{\mu,\xi}=0  &\text{on } \pa\Omega.
\end{cases}\]
Actually, the situation is even more complicated if $U_{\mu,\xi}$ is replaced with a sum of different bubbles.
The study of this nonlinear projection for the Lane-Emden system is totally new and it will be carried out in Subsection \ref{subsec:bubblepro}.
Refer to the work of Dancer, Santra and Wei \cite{DSW} where a different type of nonlinear projection was introduced for an elliptic equation with zero mass.

The third key point concerns the choice of the functional space where our system \eqref{eq:LEs} is set.
Roughly speaking, we observe that  as $\ep \to 0$ system \eqref{eq:LEs} has  the formal limit system \eqref{eq:bubble2}, which is equivalent to the single equation
$$\Delta\(|\Delta U|^{\frac{1-p}{p}}\Delta U\) = |U|^{q-1}U\ \text{in } \mathbb R^N.$$
Its solutions are achieved as extremizers of a higher-order Sobolev embedding $\dot{W}^{2,{p+1 \over p}}(\R^N) \hookrightarrow L^{q+1}(\R^N)$.
Therefore, it is natural to work in the Banach space $X_{p,q}$ introduced in \eqref{eq:Xpq}.
The Calder\'on-Zygmund estimate and the Hardy-Littlewood-Sobolev (HLS) inequality will play a crucial role
in rewriting the original problem into a more suitable one (see \eqref{eq:LEs*}) and performing the linear analysis in Section \ref{sec:lin}.

\medskip
Although the exponent $p$ in problem \eqref{eq:LEs} can take any values in $(\frac{2}{N-2}, \frac{N+2}{N-2}]$,
we will only focus on the case $p \in (1,\frac{N-1}{N-2})$.
First of all, the choice $p,q>1$ is strongly related to the method of the proof, since the reduction process can be carried out if
the nonlinearities $|u|^{q-1}u$ and $|v|^{p-1}v$ in \eqref{eq:LEs} have  superlinear growth.
Next, if $p<\frac{N}{N-2}$, system \eqref{eq:LEs} exhibits the strong nonlinear feature that the single equation \eqref{bn} does not have (as shown in Lemma \ref{lemma:wtg}).
Moreover, if $p<\frac{N-1}{N-2}$, the auxiliary function $\wth_\dxi$  defined in \eqref{eq:wth} is a regular function and this plays a crucial role in our construction.
On the other hand, if $p \in [\frac{N-1}{N-2}, \frac{N}{N-2})$, the argument in Lemma \ref{lemma:wth} yields that
the function $\wth_\dxi$ is only H\"older continuous on $\ovom$, and so its critical points may not be well-defined.
To overcome this issue, we can modify its definition (see (2.8) of \cite{CK}), but it makes the analysis much more complicated.
Finally, if $p \in [\frac{N}{N-2}, \frac{N-2}{N+2}]$, the situation is  much simpler and will not considered in the present paper.
Indeed the decay of both components of the bubble $(U_{\mu,\xi} , V_{\mu,\xi})$ in \eqref{eq:A} for large $x \in \R^N$ is
$|x|^{2-N}$ if $p \in (\frac{N}{N-2}, \frac{N-2}{N+2}]$ and $|x|^{2-N} \log|x|$ if $p = \frac{N}{N-2}$ (see \cite{HV})
and a good ansatz for the solution we want to build is just the linear projection $(PU_{\mu,\xi}, PV_{\mu,\xi}).$
In this case, the function which governs blowing-up points and rates only involves
the Green's function $G$ of $-\Delta$ in $\Omega$ with Dirichlet boundary condition and its regular part $H$ as for the single equation \cite{bn}.
We point out that if $p \in (1,\frac{N-1}{N-2})$, this function is a combination of the functions $\wtg$ and $\wth$ built up using $G$ and $H$ (see Subsection \ref{subsec:Green} and Proposition \ref{prop:exp}).

\medskip
The paper is organized as follows.

In Section \ref{sec:pre}, we introduce and prove preliminary results needed in the rest of the proof of the main theorems,
which include properties of functions $\wtg$ and $\wth$, that of the bubbles, and the nonlinear projection of the bubbles.

In Section \ref{sec:set}, we introduce the function spaces $X_{p,q}$ and examine their decompositions.
Also, we transform the original problem \eqref{eq:LEs} into a more suitable form \eqref{eq:aux1}-\eqref{eq:aux2} to apply the reduction method, define approximate solutions, and evaluate their errors.

In Section \ref{sec:lin}, we perform linear analysis, discuss on unique solvability of the auxiliary nonlinear equation \eqref{eq:aux1},
and reduce the problem to find a critical point of a function $J_{\ep}$, called the reduced energy, defined on a finite-dimensional set $\Lambda$.

In Section \ref{sec:exp}, we derive the asymptotic expansion of the reduced energy with respect to $\ep$.

In Section \ref{sec:comp}, we complete the proof of the main theorems by finding critical points of the reduced energy.

In Section \ref{sec:nod}, we briefly sketch the ideas needed to build solutions with sign-changing blow-up points.

In Section \ref{sec:subc}, we present the results on slightly subcritical problems and describe how to modify the proof to deduce them.

In Appendix \ref{sec:app}, we provide several technical computations necessary in the proof.
The primary tool is potential analysis combined with the representation formula, the Kelvin transform, and so on.

In Appendix \ref{sec:app-b}, we investigate the regularity property of our solutions, by establishing a rather general regularity result based on the HLS inequality.

\medskip \noindent \textbf{Notations.}
Here, we collect some notations which will be used throughout the paper.

\medskip \noindent - Let $B^N(x,r) = \{y \in \R^N: |y-x| < r\}$ for each $x \in \R^N$ and $r > 0$.

\medskip \noindent - $|\S^{N-1}| := 2\pi^{N/2}/\Gamma(\frac{N}{2})$ is the surface area of the $(N-1)$-dimensional unit sphere $\S^{N-1}$.

\medskip \noindent - For a function $F$ defined in $\R^N$, let $F^*$ be its Kelvin transform. Namely, we set
\[F^*(x) = \frac{1}{|x|^{N-2}} F\(\frac{x}{|x|^2}\) \quad \text{for } x \in \R^N \setminus \{0\}\]
and $F^*(0) = \lim_{|x| \to 0} F^*(x)$.

\medskip \noindent - For indices $i$ and $j$, the notation $\delta_{ij}$ denotes the Kronecker delta.

\medskip \noindent - We write $\pa_i = \frac{\pa}{\pa x_i}$, $\pa_{\xi_i} = \frac{\pa}{\pa \xi_i}$, and so on.

\medskip \noindent - For $t \in \R$, we write $t_+ = \max\{t, 0\}$.

\medskip \noindent - $C > 0$ is a generic constant that may vary from line to line.

\section{Preliminaries} \label{sec:pre}
\subsection{Properties of the Green's function and its related functions}
In this subsection, we assume that $N \ge 4$ and $p \in (\frac{2}{N-2}, \frac{N-1}{N-2})$.

\label{subsec:Green}
Let $G = G_{\Omega}$ be the Green's function of the Laplacian $-\Delta$ in $\Omega$ with the Dirichlet boundary condition.
Let also $H = H_{\Omega}: \Omega \times \Omega \to \R$ be its regular part, which solves
\[\begin{cases}
-\Delta_x H(x,y) = 0 &\text{for } x \in \Omega, \\
H(x,y) = \dfrac{\ga_N}{|x-y|^{N-2}} &\text{for } x \in \pa \Omega
\end{cases}\]
for each $y \in \Omega$ with $\ga_N := (N-2)^{-1}|\S^{N-1}|^{-1} > 0$. Then
\begin{equation}\label{eq:Green}
0 < G(x,y) = G(y,x) = \frac{\ga_N}{|x-y|^{N-2}} - H(x,y) < \frac{\ga_N}{|x-y|^{N-2}}
\end{equation}
for $(x,y) \in \Omega \times \Omega$ such that $x \ne y$.

In addition, we introduce a function $\wtg = \wtg_{\Omega}: \Omega \times \Omega \to \R$ satisfying
\begin{equation}\label{eq:wtg2}
\begin{cases}
-\Delta_x \wtg(x,y) = G^p(x,y) &\text{for } x \in \Omega,\\
\wtg(x,y) = 0 &\text{for } x \in \pa\Omega
\end{cases}
\end{equation}
given any $y \in \Omega$, and its regular part $\wth = \wth_{\Omega}: \Omega \times \Omega \to \R$ by
\begin{equation}\label{eq:wth0}
\wth(x,y) = \frac{\tga_{N,p}}{|x-y|^{(N-2)p-2}} - \wtg(x,y)
\end{equation}
where
\begin{equation}\label{eq:tga}
\tga_{N,p} := \frac{\ga_N^p}{[(N-2)p-2][N-(N-2)p]} > 0.
\end{equation}
In the next two lemmas, we discuss on the regularity and symmetry of $\wth$
for which we essentially depend on the condition that $p \in (\frac{2}{N-2}, \frac{N-1}{N-2})$.
Except for the special case $p = 1$ in which $\wtg$ becomes the Green's function of the bi-Laplacian $(-\Delta)^2$ in $\Omega$ with the Navier boundary condition,
their proofs turn out to be rather technical, because $\wtg$ involves with the nonlinear term $G^p$ as can be seen in \eqref{eq:wtg2}.
We will postpone the proof until Appendix \ref{subsec:appa2}.
\begin{lemma}\label{lemma:Green1}
For each $y \in \Omega$, the map $x \in \ovom \mapsto \nabla_x \wth(x,y)$ is continuous.
Also, for each $x \in \Omega$, the map $y \in \Omega \mapsto \nabla_y \wth(x,y)$ is continuous.
\end{lemma}
\begin{lemma}\label{lemma:Green2}
It is true that
\begin{equation}\label{eq:Green20}
p \nabla_x \wth(x,\xi)|_{x=\xi} = \nabla_y \wth(\xi,y)|_{y=\xi}.
\end{equation}
\end{lemma}

We set $\tta(\xi) = \tta_{\Omega}(\xi) = \wth_{\Omega}(\xi,\xi)$ for $\xi \in \Omega$. In the next lemma, we discuss the sign and boundary behavior of $\tta$.
\begin{lemma}
It holds that $\tta(\xi) > 0$ for all $\xi \in \Omega$.
Moreover, there exist constants $C > 0$ and $\delta > 0$ depending only on $N$, $p$ and $\Omega$ such that
\begin{equation}\label{eq:h2}
\nu_{\xi} \cdot \nabla_{\xi} \tta(\xi) \ge C\, \textnormal{dist}(\xi,\pa \Omega)^{1-(N-2)p}
\end{equation}
for $\xi \in \Omega$ with $\textnormal{dist}(\xi,\pa \Omega) < \delta$.
Here $\nu_{\xi} \in \S^{N-1}$ is the vector such that $\xi + \textnormal{dist}(\xi,\pa \Omega)\nu_{\xi} \in \pa \Omega$.
\end{lemma}
\begin{proof}
For each $\xi \in \Omega$, we have that $-\Delta_x \wth(x,\xi) > 0$ for $x \in \Omega$ and $\wth(x,\xi) > 0$ for $x \in \pa \Omega$.
The maximum principle implies that $\wth(x,\xi) > 0$ for $x \in \Omega$, and in particular, $\tta(\xi) > 0$ for $\xi \in \Omega$.

On the other hand, by Lemma \ref{lemma:Green2},
\begin{equation}\label{eq:h21}
\nabla_{\xi} \tta(\xi) = (p+1) \nabla_x \wth(x,\xi)|_{x = \xi} \quad \text{for any } \xi \in \Omega.
\end{equation}
Besides, Proposition 2.4 of \cite{CK} states that
\begin{equation}\label{eq:h22}
\nu_x \cdot \nabla_x \wth(x,\xi)|_{x = \xi} \ge C\, \textnormal{dist}(\xi,\pa \Omega)^{1-(N-2)p}
\end{equation}
for $N \ge 5$. A closer inspection shows that its proof also works for $N = 4$ and $p \in (\frac{2}{N-2}, \frac{N-1}{N-2})$.
Now, \eqref{eq:h2} follows from \eqref{eq:h21} and \eqref{eq:h22}.
\end{proof}

\subsection{Properties of the bubbles}\label{subsec:bubble}
Let $N \ge 3$, $(p,q)$ be a pair of positive numbers such that $p \in (\frac{2}{N-2}, \frac{N}{N-2})$ and satisfy \eqref{eq:hyper}, and $(U,V)$ a positive ground state solution to
\begin{equation}\label{eq:bubble}
\begin{cases}
-\Delta U = |V|^{p-1}V \quad{\text{in }} \R^N,\\
-\Delta V = |U|^{q-1}U \quad{\text{in }} \R^N,\\
(U, V) \in \dot{W}^{2,{p+1 \over p}}(\R^N) \times \dot{W}^{2,{q+1 \over q}}(\R^N)
\end{cases}
\end{equation}
found by Lions in Corollary I.2 of \cite{Li}. According to Alvino, Lions and Trombetti \cite{ALT} (see also Corollary I.2 of \cite{Li}),
it is radially symmetric and decreasing in the radial variable after a suitable translation.
Moreover, the results due to Wang in Lemma 3.2 of \cite{Wa} and Hulshof and Van der Vorst in Theorem 1 of \cite{HV} tell us that
there is the unique positive ground state $(U_{1,0}(x), V_{1,0}(x))$ of \eqref{eq:LEs} such that $U_{1,0}(0) = 1$,
and the family of functions $\{(U_{\mu,\xi}(x), V_{\mu,\xi}(x))\}$ given by
\begin{multline}\label{eq:Umx}
(U_{\mu,\xi}(x), V_{\mu,\xi}(x)) = \(\mu^{-{N \over q+1}} U_{1,0}(\mu^{-1}(x-\xi)),\, \mu^{-{N \over p+1}} V_{1,0}(\mu^{-1}(x-\xi))\) \\
\quad \text{for any } \mu > 0,\, \xi \in \R^N
\end{multline}
exhausts all the positive ground states of \eqref{eq:LEs}.

In the proof of Theorem \ref{thm:main1} and Theorem \ref{thm:main2}, we need the the sharp asymptotic behavior of ground states to \eqref{eq:LEs} and the non-degeneracy for \eqref{eq:LEs} at each ground state.
\begin{lemma}[Theorem 2 in \cite{HV}]\label{lemma:dec}
There exist positive constants $a_{N,p}$ and $b_{N,p}$ depending only on $N$ and $p$ such that
\begin{equation}\label{eq:dec1}
\begin{cases}
\lim\limits_{r \to \infty} r^{(N-2)p-2} U_{1,0}(r) = a_{N,p}, \\
\lim\limits_{r \to \infty} r^{N-2} V_{1,0}(r) = b_{N,p}
\end{cases}
\end{equation}
where we wrote $U_{1,0}(x) = U_{1,0}(|x|)$, $V_{1,0}(x) = V_{1,0}(|x|)$ and $r = |x|$ by abusing notations. Furthermore,
\begin{equation}\label{eq:dec2}
b_{N,p}^p = a_{N,p}[(N-2)p-2][N-(N-2)p].
\end{equation}
\end{lemma}
\begin{lemma}[Theorem 1 in \cite{FKP}]\label{lemma:nondeg}
Set
\[(\Psi_{1,0}^0, \Phi_{1,0}^0) = \(x \cdot \nabla U_{1,0} + \frac{NU_{1,0}}{q+1},\, x \cdot \nabla V_{1,0} + \frac{NV_{1,0}}{p+1}\)\]
and
\[(\Psi_{1,0}^{\ell}, \Phi_{1,0}^{\ell}) = \(\pa_{\ell} U_{1,0}, \pa_{\ell} V_{1,0}\) \quad \text{for } \ell = 1, \cdots, N.\]
Then the space of solutions to the linear system
\begin{equation}\label{eq:lin-sys}
\begin{cases}
-\Delta \Psi = pV_{1,0}^{p-1} \Phi \quad{\text{in }} \R^N,\\
-\Delta \Phi = qU_{1,0}^{q-1} \Psi \quad{\text{in }} \R^N,\\
(\Psi, \Phi) \in \dot{W}^{2,{p+1 \over p}}(\R^N) \times \dot{W}^{2,{q+1 \over q}}(\R^N)
\end{cases}
\end{equation}
is spanned by
\[\left\{(\Psi_{1,0}^0, \Phi_{1,0}^0), (\Psi_{1,0}^1, \Phi_{1,0}^1), \cdots, (\Psi_{1,0}^N, \Phi_{1,0}^N) \right\}.\]
\end{lemma}

Before closing this subsection, we collect some corollaries of Lemmas \ref{lemma:dec} and \ref{lemma:nondeg}.
The first two results are refinements of Lemma \ref{lemma:dec}. Because of technical reasons, we put off their proofs until Appendix \ref{subsec:appa3}.
\begin{cor}\label{cor:dec1}
Given any $\zeta \in (0,1)$, it holds that
\begin{equation}\label{eq:dec3}
\left| V_{1,0}(x) - \frac{b_{N,p}}{|x|^{N-2}} \right| = O\(\frac{1}{|x|^{N-1}}\)
\end{equation}
and
\begin{equation}\label{eq:dec4}
\left|\nabla V_{1,0}(x) + (N-2)b_{N,p}\frac{x}{|x|^N} \right| = O\(\frac{1}{|x|^{N-\zeta}}\)
\end{equation}
for $|x| \ge 1$.
\end{cor}
\begin{cor}\label{cor:dec2}
Assume further that $p \in (1,\frac{N-1}{N-2})$. It holds that
\begin{equation}\label{eq:dec5}
\left| U_{1,0}(x)-\frac{a_{N,p}}{|x|^{(N-2)p-2}} \right| = O\(\frac{1}{|x|^{(N-2)p-1}}\)
\end{equation}
and
\begin{equation}\label{eq:dec6}
\left|\nabla U_{1,0}(x) + ((N-2)p-2)a_{N,p} \frac{x}{|x|^{(N-2)p}} \right| = O\(\dfrac{1}{|x|^{\kappa_0}}\)
\end{equation}
where
\[\kappa_0 := \min\{N-2,(N-1)p-2\} > (N-2)p-1\]
for $|x| \ge 1$.
\end{cor}
\noindent The bounds in \eqref{eq:dec4} and \eqref{eq:dec6} seem not optimal and one may improve them.
Because they are enough for our purpose, we will not pursue it.

Thanks to the scaling invariance of \eqref{eq:bubble} described in \eqref{eq:Umx}, we can rewrite Lemma \ref{lemma:nondeg} as follows.
\begin{cor}
For any $\mu > 0$ and $\xi = (\xi_1, \cdots, \xi_N) \in \R^N$, we set
\[(\Psi_{\mu,\xi}^0, \Phi_{\mu,\xi}^0) = -\(\pa_{\mu} U_{\mu,\xi}, \pa_{\mu} V_{\mu,\xi}\)
= \(\mu^{-{N \over q+1}-1} \Psi_{1,0}^0(\mu^{-1}(x-\xi)), \mu^{-{N \over p+1}-1} \Phi_{1,0}^0(\mu^{-1}(x-\xi))\)\]
and
\[(\Psi_{\mu,\xi}^{\ell}, \Phi_{\mu,\xi}^{\ell}) = -\(\pa_{\xi_{\ell}} U_{\mu,\xi}, \pa_{\xi_{\ell}} V_{\mu,\xi}\)
= \(\mu^{-{N \over q+1}-1} \Psi_{1,0}^{\ell}(\mu^{-1}(x-\xi)), \mu^{-{N \over p+1}-1} \Phi_{1,0}^{\ell}(\mu^{-1}(x-\xi))\)\]
for $\ell = 1, \cdots, N$. Then the space of solutions to the linear system
\[\begin{cases}
-\Delta \Psi = pV_{\mu,\xi}^{p-1} \Phi \quad{\text{in }} \R^N,\\
-\Delta \Phi = qU_{\mu,\xi}^{q-1} \Psi \quad{\text{in }} \R^N,\\
(\Psi, \Phi) \in \dot{W}^{2,{p+1 \over p}}(\R^N) \times \dot{W}^{2,{q+1 \over q}}(\R^N).
\end{cases}\]
is spanned by
\[\left\{(\Psi_{\mu,\xi}^0, \Phi_{\mu,\xi}^0), (\Psi_{\mu,\xi}^1, \Phi_{\mu,\xi}^1), \cdots, (\Psi_{\mu,\xi}^N, \Phi_{\mu,\xi}^N) \right\}.\]
\end{cor}

\subsection{Projection of the bubbles}\label{subsec:bubblepro}
In this subsection, we assume that $N \ge 4$ and $p \in (\frac{2}{N-2}, \frac{N-1}{N-2})$.

We fix $k \in \N$, and write $\mu_i = \mu d_i$ for a small number $\mu > 0$ and $i = 1, \cdots, k$.
Moreover, given numbers $\delta_1, \delta_2 \in (0,1)$ small enough, we  define the configuration space $\Lambda$ by
\begin{multline}\label{eq:Lambda}
\Lambda = \left\{(\mbd, \mbxi): \mbd = (d_1, \cdots, d_k) \in (\delta_1, \delta_1^{-1})^k,\, \mbxi = (\xi_1, \cdots, \xi_k) \in \Omega^k,\, \right. \\
\left. \text{dist}(\xi_i, \pa\Omega) \ge \delta_2 \text{ for } 1 \le i \le k,\,
\text{dist}(\xi_i, \xi_j) \ge \delta_2 \text{ for } 1 \le i \ne j \le k \right\}.
\end{multline}
For a fixed parameter $(\mbd, \mbxi) \in \Lambda$, we denote
\[(U_i, V_i) = (U_{\mu_i, \xi_i}, V_{\mu_i, \xi_i}) \quad \text{for } i = 1, \cdots, k\]
and let the pair $(PU_i, PV_i)$ be the unique smooth solution of the system
\begin{equation}\label{eq:PUPV}
\begin{cases}
-\Delta PU_i = V_i^p &\text{in } \Omega,\\
-\Delta PV_i = U_i^p &\text{in } \Omega,\\
PU_i = PV_i = 0 &\text{on } \pa \Omega
\end{cases}
\end{equation}
for $i = 1, \cdots, k$. A standard comparison argument based on \eqref{eq:dec1} yields
\begin{lemma}\label{lemma:PUPV}
Let $\whh: \Omega \times \Omega \to \R$ be a smooth function such that
\[\begin{cases}
-\Delta_x \whh(x,y) = 0 &\text{for } x \in \Omega, \\
\whh(x,y) = \dfrac{1}{|x-y|^{(N-2)p-2}} &\text{for } x \in \pa \Omega
\end{cases}\]
given any $y \in \Omega$. If $i = 1, \cdots, k$, then we have
\begin{equation}\label{eq:PU}
PU_i(x) = U_i(x) - a_{N,p}\, \mu_i^{Np \over q+1} \whh(x,\xi_i) + o(\mu^{Np \over q+1})
\end{equation}
and
\begin{equation}\label{eq:PV}
PV_i(x) = V_i(x) - \(\frac{b_{N,p}}{\ga_N}\) \mu_i^{N \over q+1} H(x,\xi_i) + o(\mu^{N \over q+1})
\end{equation}
for $x \in \Omega$.
\end{lemma}
Analogously, we denote
\[(\Psi_{i\ell}, \Phi_{i\ell}) = (\Psi_{\mu_i,\xi_i}^{\ell}, \Phi_{\mu_i,\xi_i}^{\ell}) \quad \text{for } i = 1, \cdots, k \text{ and } \ell = 0, \cdots, N,\]
and let the pair $(P\Psi_{i\ell}, P\Phi_{i\ell})$ be the unique smooth solution of the system
\begin{equation}\label{eq:PsPh}
\begin{cases}
-\Delta P\Psi_{i\ell} = pV_i^{p-1}\Phi_{i\ell} &\text{in } \Omega,\\
-\Delta P\Phi_{i\ell} = qU_i^{q-1}\Psi_{i\ell} &\text{in } \Omega,\\
P\Psi_{i\ell} = P\Phi_{i\ell} = 0 &\text{on } \pa \Omega
\end{cases}
\end{equation}
for $i = 1, \cdots, k$ and $\ell = 0, \cdots, N$. Applying a comparison argument together with Corollaries \ref{cor:dec1} and \ref{cor:dec2}, we observe
\begin{lemma}\label{lemma:PsPh}
If $i = 1, \cdots, k$ and $\ell = 0, \cdots, N$, it holds that
\[P\Psi_{i\ell}(x) = \begin{cases}
\Psi_{i\ell}(x) + a_{N,p}\, \mu_i^{{Np \over q+1}-1} \whh(x,\xi_i) + o(\mu^{{Np \over q+1}-1}) &\text{for } \ell = 0,\\
\Psi_{i\ell}(x) + a_{N,p}\, \mu_i^{Np \over q+1} \pa_{\xi,\ell} \whh(x,\xi_i) + o(\mu^{Np \over q+1}) &\text{for } \ell = 1, \cdots, N,
\end{cases}\]
and
\begin{equation}\label{eq:PPhi}
P\Phi_{i\ell}(x) = \begin{cases}
\displaystyle \Phi_{i\ell}(x) + \(\frac{b_{N,p}}{\ga_N}\) \mu_i^{{N \over q+1}-1} H(x,\xi_i) + o(\mu^{{N \over q+1} - 1}) &\text{for } \ell = 0,\\
\displaystyle \Phi_{i\ell}(x) + \(\frac{b_{N,p}}{\ga_N}\) \mu_i^{N \over q+1} \pa_{\xi,\ell} H(x,\xi_i) + o(\mu^{N \over q+1}) &\text{for } \ell = 1, \cdots, N,
\end{cases}
\end{equation}
for $x \in \Omega$. Here, $\pa_{\xi,\ell} \whh(x,\xi)$ and $\pa_{\xi,\ell} H(x,\xi)$ stand for the $\ell$-th components of $\nabla_{\xi} \whh(x,\xi)$ and $\nabla_{\xi} H(x,\xi)$, respectively.
\end{lemma}

We next examine the function $\mbP U_\dxi$ defined as the smooth solution of the equation
\begin{equation}\label{eq:mbP}
\begin{cases}
\displaystyle -\Delta \mbP U_\dxi = \(\sum_{i=1}^k PV_i\)^p &\text{in } \Omega,\\
\mbP U_\dxi = 0 &\text{on } \pa \Omega.
\end{cases}
\end{equation}
For this aim, we need two auxiliary functions: Let $\wtg_\dxi: \Omega \to \R$ be the solution of
\begin{equation}\label{eq:wtg}
\begin{cases}
\displaystyle -\Delta \wtg_\dxi(x) = \(\sum_{i=1}^k d_i^{N \over q+1} G(x,\xi_i)\)^p &\text{for } x \in \Omega, \\
\wtg_\dxi = 0 &\text{for } x \in \pa \Omega,
\end{cases}
\end{equation}
and $\wth_\dxi: \Omega \to R$ be its regular part given as
\begin{equation}\label{eq:wth}
\wth_\dxi(x) = \sum_{i=1}^k d_i^{Np \over q+1} \frac{\tga_{N,p}}{|x-\xi_i|^{(N-2)p-2}} - \wtg_\dxi(x) \quad \text{for } x \in \Omega.
\end{equation}
As a preliminary result, we concern the regularity of $\wth_\dxi$.
\begin{lemma}\label{lemma:wth}
The norm $\|\wth_\dxi\|_{C^{1,\sigma}(\overline{\Omega})}$ is uniformly bounded in $\Lambda$ for some $\sigma \in (0,1)$.
\end{lemma}
\begin{proof}
By \eqref{eq:wtg}, \eqref{eq:wth} and \eqref{eq:tga},
\begin{equation}\label{eq:wth2}
\begin{cases}
\displaystyle -\Delta \wth_\dxi(x) = \sum_{i=1}^k \(\frac{d_i^{N \over q+1} \ga_N}{|x-\xi_i|^{N-2}}\)^p - \(\sum_{i=1}^k d_i^{N \over q+1} G(x,\xi_i)\)^p &\text{for } x \in \Omega, \\
\displaystyle \wth_\dxi(x) = \sum_{i=1}^k d_i^{Np \over q+1} \frac{\tga_{N,p}}{|x-\xi_i|^{(N-2)p-2}} &\text{for } x \in \pa \Omega.
\end{cases}
\end{equation}
Fix $i = 1, \cdots, k$ and a number $\rho \in (0,\delta_2/3)$. Then \eqref{eq:Green} leads us to
\begin{equation}\label{eq:wthb}
\left| \Delta \wth_\dxi(x) \right| \le C\(\frac{1}{|x-\xi_i|^{(N-2)(p-1)}} + 1\) \quad \text{for } x \in B^N(\xi_i,\rho)
\end{equation}
for a constant $C > 0$ depending only on $N$, $p$, $\delta_1$ and $\delta_2$.
Estimate \eqref{eq:wthb} and the condition that $p < \frac{N-1}{N-2}$ yield
\[\|\Delta \wth_\dxi\|_{L^t(\Omega)} + \|\wth_\dxi\|_{C^{1,\sigma}(\pa \Omega)} \le C\]
for some $t > N$, uniformly in $\Lambda$.
By elliptic regularity, the lemma follows.
\end{proof}
\noindent Because the right-hand side of the first equation in \eqref{eq:mbP} contains a power nonlinearity,
it is difficult to find the expansion of $\mbP U_\dxi$ by applying only the comparison argument, except the special case $p = 1$.
We will overcome this technical issue by performing potential analysis.
\begin{lemma}\label{lemma:wtg}
We have that
\begin{equation}\label{eq:bPU}
\mbP U_\dxi(x) = \sum_{i=1}^k U_i(x) - \mu^{Np \over q+1} \({b_{N,p} \over \ga_N}\)^p \wth_\dxi(x) + o(\mu^{Np \over q+1})
\end{equation}
for $x \in \Omega$.
\end{lemma}
\begin{proof}
Let $\vph_\dxi: \Omega \to \R$ be the solution of
\begin{equation}\label{eq:vph}
\begin{cases}
\displaystyle -\Delta \vph_\dxi = \(\sum_{i=1}^k d_i^{N \over q+1} G(\cdot,\xi_i)\)^p - \sum_{i=1}^k \(\frac{d_i^{N \over q+1} \ga_N}{|\cdot-\xi_i|^{N-2}}\)^p &\text{in } \Omega, \\
\vph_\dxi = 0 &\text{on } \pa \Omega.
\end{cases}
\end{equation}
We will first show that
\begin{equation}\label{eq:bPU2}
\mbP U_\dxi(x) = \sum_{i=1}^k PU_i(x) + \mu^{Np \over q+1} \({b_{N,p} \over \ga_N}\)^p \vph_\dxi(x) + o(\mu^{Np \over q+1})
\end{equation}
for $x \in \Omega$.

By the representation formula and \eqref{eq:PV},
\begin{align*}
\mbP U_\dxi(x) - \sum_{i=1}^k PU_i(x) 
&= \sum_{l=1}^k \int_{B^N(\xi_l, \mu_l^{\kappa_1})} G(x,y) \left[\(\sum_{i=1}^k PV_i\)^p - \sum_{i=1}^k V_i^p \right](y) dy \\
&\ + \int_{\Omega \setminus (\cup_{l=1}^k B^N(\xi_l, \mu_l^{\kappa_1}))} G(x,y) \left[\(\sum_{i=1}^k PV_i\)^p - \sum_{i=1}^k V_i^p \right](y) dy \\
&=: \sum_{l=1}^k I_{1l}(x) + I_2(x)
\end{align*}
where $\kappa_1 \in (0,1)$ is a number which will be chosen later.

Let us estimate $I_{1l}$ for $l = 1, \cdots, k$. By \eqref{eq:dec3} and \eqref{eq:PV},
\begin{equation}\label{eq:ineq0}
|I_{1l}(x)| \le C \int_{B^N(\xi_l, \mu_l^{\kappa_1})} \frac{1}{|x-y|^{N-2}} \(\mu^{N \over q+1} V_l^{p-1}(y) + \mu^{Np \over q+1}\) dy.
\end{equation}
Also, it holds that
\begin{equation}\label{eq:ineq1}
\mu^{N \over q+1} \int_{B^N(\xi_l, \mu_l^{\kappa_1})} \frac{1}{|x-y|^{N-2}} V_l^{p-1}(y) dy \le C\mu^{Np \over q+1} \mu^{(N-(N-2)p)\kappa_1}
\end{equation}
and
\begin{equation}\label{eq:ineq2}
\mu^{Np \over q+1} \int_{B^N(\xi_l, \mu_l^{\kappa_1})} \frac{1}{|x-y|^{N-2}} dy \le C\mu^{Np \over q+1} \mu^{2\kappa_1}.
\end{equation}
Because the proof of the above inequalities is rather involved, we postpone it to Appendix \ref{subsec:ineq}.
From \eqref{eq:ineq0}-\eqref{eq:ineq2}, we get $I_{1l} = o(\mu^{Np \over q+1})$.

We next evaluate $I_2$. According to \eqref{eq:Green}, \eqref{eq:dec3} and \eqref{eq:PV},
\[V_i(x) = \mu_i^{N \over q+1} \left[\frac{b_{N,p}}{|x-\xi_i|^{N-2}} + O\(\frac{\mu_i}{|x-\xi_i|^{N-1}}\)\right]\]
and
\[PV_i(x) = \mu_i^{N \over q+1} \({b_{N,p} \over \ga_N}\) G(x,\xi_i) + O\(\frac{\mu_i^{{N \over q+1}+1}}{|x-\xi_i|^{N-1}}\) + o(\mu_i^{N \over q+1})\]
in $\Omega \setminus B^N(\xi_i, \mu_i^{\kappa_1})$.
Owing to the above estimates and the dominated convergence theorem,
\begin{align*}
&\ I_2(x) \\
&= \mu^{Np \over q+1} \({b_{N,p} \over \ga_N}\)^p \int_{\Omega \setminus (\cup_{l=1}^k B^N(\xi_l, \mu_l^{\kappa_1}))}
G(x,y) \left[ \(\sum_{i=1}^k d_i^{N \over q+1} G(y,\xi_i)\)^p
- \sum_{i=1}^k \(\frac{d_i^{N \over q+1} \ga_N}{|y-\xi_i|^{N-2}}\)^p \right. \\
&\hspace{40pt} \left. + O\(\sum_{i=1}^k \frac{\mu}{|x-\xi_i|^{(N-2)p+1}} + \frac{\mu^p}{|x-\xi_i|^{(N-1)p}}\)
+ o\(\sum_{i=1}^k \frac{1}{|y-\xi_i|^{(N-2)(p-1)}} + 1\) \right] dy \\
&= \mu^{Np \over q+1} \({b_{N,p} \over \ga_N}\)^p \int_{\Omega}
G(x,y) \left[ \(\sum_{i=1}^k d_i^{N \over q+1} G(y,\xi_i)\)^p
- \sum_{i=1}^k \(\frac{d_i^{N \over q+1} \ga_N}{|y-\xi_i|^{N-2}}\)^p \right] dy + o(\mu^{Np \over q+1}) \\
&= \mu^{Np \over q+1} \({b_{N,p} \over \ga_N}\)^p \vph_\dxi(x) + o(\mu^{Np \over q+1})
\end{align*}
provided that $0 < \kappa_1 < \frac{1}{(N-2)p+1}$. 
As a result, expansion \eqref{eq:bPU2} is valid.

We now deduce \eqref{eq:bPU}. Combining \eqref{eq:bPU2} and \eqref{eq:PU}, we discover
\[\mbP U_\dxi(x) = \sum_{i=1}^k U_i(x) + \mu^{Np \over q+1} \left[ \({b_{N,p} \over \ga_N}\)^p \vph_\dxi(x) - a_{N,p} \sum_{i=1}^k d_i^{Np \over q+1} \whh(x,\xi_i)\right] + o(\mu^{Np \over q+1}).\]
By virtue of \eqref{eq:vph}, the term in the above parenthesis can be rewritten as $b_{N,p}^p \ga_N^{-p} \wtg_\dxi + \Xi$ where $\Xi$ is the solution of
\[\begin{cases}
\displaystyle -\Delta \Xi(x) = - b_{N,p}^p \sum_{i=1}^k \frac{d_i^{Np \over q+1}}{|x-\xi_i|^{(N-2)p}} &\text{in } \Omega, \\
\displaystyle \Xi(x) = - a_{N,p} \sum_{i=1}^k \frac{d_i^{Np \over q+1}}{|x-\xi_i|^{(N-2)p-2}} &\text{on } \pa \Omega.
\end{cases}\]
It follows from \eqref{eq:dec2} and \eqref{eq:tga} that
\[\Xi(x) = - \({b_{N,p} \over \ga_N}\)^p \sum_{i=1}^k d_i^{Np \over q+1} \frac{\tga_{N,p}}{|x-\xi_i|^{(N-2)p-2}} \quad \text{in } \Omega.\]
Therefore, $\wth_\dxi = - (b_{N,p}^p \ga_N^{-p} \wtg_\dxi + \Xi)$ in $\Omega$ and the assertion is proved.
\end{proof}
\begin{rmk}
If either $k = 1$ or $p = 1$, then \eqref{eq:bPU} is reduced to
\[\mbP U_\dxi(x) = \sum_{i=1}^k U_i - \mu^{Np \over q+1} \({b_{N,p} \over \ga_N}\)^p \sum_{i=1}^k d_i^{Np \over q+1} \wth(x,\xi_i) + o(\mu^{Np \over q+1})\]
for $x \in \Omega$.
\end{rmk}

\section{Setting of the problem} \label{sec:set}
From Section \ref{sec:set} to \ref{sec:comp}, we assume that $N \ge 8$ and $p \in (1, \frac{N-1}{N-2})$ unless otherwise stated.

\subsection{Function spaces}
Let
\begin{equation}\label{eq:pq-star}
\frac{1}{p^*} = \frac{p}{p+1} - \frac{1}{N} = \frac{1}{q+1} + \frac{1}{N}
\quad \text{and} \quad
\frac{1}{q^*} = \frac{q}{q+1} - \frac{1}{N} = \frac{1}{p+1} + \frac{1}{N}
\end{equation}
so that $p^*$ and $q^*$ are H\"older conjugates of each other.
It is simple to check that $q^* \in (1,2)$. The Sobolev embedding theorem shows
\begin{equation}\label{eq:Sob}
\begin{cases}
\dot{W}^{2,\frac{p+1}{p}}(\R^N) \hookrightarrow \dot{W}^{1,p^*}(\R^N) \hookrightarrow L^{q+1}(\R^N),\\
\dot{W}^{2,\frac{q+1}{q}}(\R^N) \hookrightarrow \dot{W}^{1,q^*}(\R^N) \hookrightarrow L^{p+1}(\R^N).
\end{cases}
\end{equation}
Having this fact in mind, we introduce the Banach space
\begin{equation}\label{eq:Xpq}
X_{p,q} := \(W^{2,\frac{p+1}{p}}(\Omega) \cap W^{1,p^*}_0(\Omega)\) \times \(W^{2,\frac{q+1}{q}}(\Omega) \cap W^{1,q^*}_0(\Omega)\)
\end{equation}
equipped with the norm
\[\|(u,v)\|_{X_{p,q}} := \|\Delta u\|_{L^{\frac{p+1}{p}}(\Omega)} + \|\Delta v\|_{L^{\frac{q+1}{q}}(\Omega)}.\]

We denote by $\mci^*$ the formal adjoint operator of the embedding $\mci: X_{q,p} \hookrightarrow L^{p+1}(\Omega) \times L^{q+1}(\Omega)$.
In other words, we say that $\mci^*(w,z) = (u,v)$ if and only if
\[\begin{cases}
-\Delta u = w &\text{in } \Omega,\\
-\Delta v = z &\text{in } \Omega,\\
u = v = 0 &\text{on } \pa\Omega,
\end{cases}
\quad \text{or equivalently,} \quad
\begin{cases}
\displaystyle u(x) = \int_{\Omega} G(x,y) w(y)\, dy,\\
\displaystyle v(x) = \int_{\Omega} G(x,y) z(y)\, dy
\end{cases}
\text{for } x \in \Omega.\]
By the Calder\'on-Zygmund estimate, the operator $\mci^*$ maps $L^{p+1 \over p}(\Omega) \times L^{q+1 \over q}(\Omega)$ to $X_{p,q}$. We will rewrite problem \eqref{eq:LEs} as
\begin{equation}\label{eq:LEs*}
(u,v) = \mci^*\(|v|^{p-1}v + \ep (\alpha u + \beta_1 v), |u|^{q-1}u + \ep (\beta_2 u + \alpha v)\).
\end{equation}

\medskip
From now on, we denote $X = X_{p,q}$, and assume that $\mu > 0$ is a quantity determined by $\ep$, $N$ and $p$ such that $\mu \to 0$ as $\ep \to 0$,
whose precise value will be given in Subsection \ref{subsec:proof11}.
Let $Y_\dxi$ and $Z_\dxi$ be two subspaces of $X$ given as
\[Y_\dxi = \text{span} \left\{\(P\Psi_{i\ell}, P\Phi_{i\ell}\): i = 1, \cdots, k \text{ and } \ell = 0, \cdots, N \right\}\]
and
\begin{multline*}
Z_\dxi = \left\{(\Psi,\Phi) \in X: \int_{\Omega} \(p V_i^{p-1} \Phi_{i\ell} \Phi + q U_i^{q-1} \Psi_{i\ell} \Psi\) dx = 0 \right. \\
\left. \quad \text{for } i = 1, \cdots, k \text{ and } \ell = 0, \cdots, N \right\}.
\end{multline*}
\begin{lemma}\label{lemma:lin1}
There exists a small number $\ep_0 > 0$ such that if $\ep \in (0,\ep_0)$, then the subspaces of $Y_\dxi$ and $Z_\dxi$ of the Banach space $X$ are topological complements of each other.
In short, $X = Y_\dxi \oplus Z_\dxi$.
\end{lemma}
\begin{proof}
We shall show that for each $(\Psi,\Phi) \in X$, there exist unique pairs $(\Psi_0,\Phi_0) \in Z_\dxi$
and coefficients $(c_{10}, c_{11}, \cdots, c_{1N}, c_{20}, \cdots, c_{(k-1)N}, c_{k0} \cdots, c_{kN}) \in \R^{k(N+1)}$ such that
\begin{equation}\label{eq:lin11}
(\Psi,\Phi) = (\Psi_0,\Phi_0) + \sum_{i=0}^k \sum_{\ell=0}^N c_{i\ell} \(P\Psi_{i\ell}, P\Phi_{i\ell}\).
\end{equation}
The requirement that $(\Psi_0,\Phi_0) \in Z_\dxi$ is equivalent to demanding
\begin{multline}\label{eq:lin12}
\sum_{j=1}^k \sum_{m=0}^N c_{jm} \int_{\Omega} \(p V_i^{p-1} \Phi_{i\ell} \cdot P\Phi_{jm} + q U_i^{q-1} \Psi_{i\ell} \cdot P\Psi_{jm} \) \\
= \int_{\Omega} \( p V_i^{p-1} \Phi_{i\ell} \Phi + q U_i^{q-1} \Psi_{i\ell} \Psi \)
\end{multline}
for all $i = 1, \cdots, k$ and $\ell = 0, \cdots, N$.

Let us estimate the integral on the left-hand side of \eqref{eq:lin12}.
By employing \eqref{eq:PPhi} and Corollary \ref{cor:dec1}, and then taking $y = \mu_i^{-1}(x-\xi_i)$, we get
\begin{align*}
&\ \int_{\Omega} p V_i^{p-1} \Phi_{i\ell} \cdot P \Phi_{jm} \\
&= \int_{\R^N} p V_i^{p-1}(x) \Phi_{i\ell}(x) \Phi_{jm}(x) dx + O(\mu^{-2} \mu^{(N-2)p-2}) \\
&= (\mu_i\mu_j)^{-1} \(\frac{d_i}{d_j}\)^{N \over p+1}
\int_{\R^N} p V_{1,0}^{p-1}(y) \Phi_{1,0}^{\ell}(y) \Phi_{1,0}^m\(\frac{d_i}{d_j} \(y +\frac{\xi_i-\xi_j}{\mu_i}\)\) dy + O(\mu^{-2} \mu^{(N-2)p-2}).
\end{align*}
Thus, for $i = j$, we have
\begin{equation}\label{eq:lin131}
\begin{aligned}
\int_{\Omega} p V_i^{p-1} \Phi_{i\ell} \cdot P \Phi_{jm}
&= \mu^{-2}d_i^{-2} \int_{\R^N} p V_{1,0}^{p-1} \Phi_{1,0}^{\ell} \Phi_{1,0}^m + O(\mu^{-2} \mu^{(N-2)p-2}) \\
&= \mu^{-2} \delta_{\ell m} d_i^{-2} \int_{\R^N} p V_{1,0}^{p-1} (\Phi_{1,0}^{\ell})^2 + o(\mu^{-2}).
\end{aligned}
\end{equation}
If $i \ne j$, a suitable choice of $R > 0$ large and $\rho > 0$ small yields 
\begin{align}
\left| \int_{\Omega} p V_i^{p-1} \Phi_{i\ell} \cdot P \Phi_{jm} \right| &\le C \mu^{-2}
\left[ \int_{B^N(0,R\mu^{-1}) \setminus B^N(\mu_i^{-1}(\xi_j-\xi_i), \rho\mu^{-1})} \frac{1}{1+|y|^{(N-2)p}} \mu^{N-2} dy \right. \nonumber \\
&\hspace{40pt} + \int_{B^N(0, d_id_j^{-1} \rho\mu^{-1})} \mu^{(N-2)p} \Phi_{1,0}^m(y) dy \label{eq:lin132}\\
&\hspace{40pt} \left. + \int_{\R^N \setminus B^N(0,R\mu^{-1})} \frac{1}{|y|^{(N-2)(p+1)}} dy + O(\mu^{(N-2)p-2}) \right] \nonumber \\
&= O(\mu^{-2} \mu^{(N-2)p-2}) = o(\mu^{-2}). \nonumber
\end{align}
Applying Corollary \ref{cor:dec2} and Lemma \ref{lemma:alg}, we can perform a similar analysis to deduce
\begin{equation}\label{eq:lin14}
\int_{\Omega} q U_i^{q-1} \Psi_{i\ell} \cdot P\Psi_{jm} = \mu^{-2} \delta_{\ell m} d_i^{-2} \int_{\R^N} q U_{1,0}^{q-1} (\Psi_{1,0}^{\ell})^2 + o(\mu^{-2}) 
\end{equation}
for $i = j$, and
\begin{equation}\label{eq:lin15}
\int_{\Omega} q U_i^{q-1} \Psi_{i\ell} \cdot P\Psi_{jm} 
= O\(\mu^{-2} \(\mu^{(N-2)p-2} + \mu^{((N-2)p-2)(q+1)-N}\)\)
= o(\mu^{-2})
\end{equation}
for $i \ne j$.

By plugging \eqref{eq:lin131}-\eqref{eq:lin15} into \eqref{eq:lin12}, we see
\begin{multline}\label{eq:lin16}
\sum_{j=1}^k \sum_{m=0}^N c_{jm} \left[\delta_{ij} \delta_{\ell m} d_i^{-2} \int_{\R^N} \left\{p V_{1,0}^{p-1} (\Phi_{1,0}^{\ell})^2 + q U_{1,0}^{q-1} (\Psi_{1,0}^{\ell})^2\right\} + o(1) \right] \\
= \mu^2 \int_{\Omega} \( p V_i^{p-1} \Phi_{i\ell} \Phi + q U_i^{q-1} \Psi_{i\ell} \Psi \)
\end{multline}
for all $i = 1, \cdots, k$ and $\ell = 0, \cdots, N$,
from which the coefficients $c_{i\ell}$'s are uniquely determined.
By virtue of \eqref{eq:lin11}, so is $(\Psi_0, \Phi_0)$.

\medskip
On the other hand, $Y_\dxi$ and $Z_\dxi$ are clearly closed subspaces of $X$. Therefore, they are topological complements of each other.
\end{proof}
\noindent Inspecting the above proof, we also obtain
\begin{cor}\label{cor:lin2}
Let $\Pi_\dxi: X \to Y_\dxi$ be a map given as
\begin{equation}\label{eq:lin20}
\Pi_\dxi(\Psi,\Phi) = \sum_{i=0}^k \sum_{\ell=0}^N c_{i\ell} \(P\Psi_{i\ell}, P\Phi_{i\ell}\)
\end{equation}
where the coefficients $c_{i\ell}$'s are determined by \eqref{eq:lin12}. Then $\Pi_\dxi$ is a well-defined linear operator.
Furthermore, there exists a universal constant $C > 0$ such that
\begin{equation}\label{eq:lin21}
\|\Pi_\dxi\|_{X \to Y_\dxi} := \sup_{\|(\Psi,\Phi)\|_X \le 1} \|\Pi_\dxi(\Psi,\Phi)\|_X \le C
\end{equation}
for all $\ep \in (0,\ep_0)$ (reducing the value of $\ep_0 > 0$ if needed).
We mean by universality that $C > 0$ is independent of $\ep \in (0,\ep_0)$ and $(\dxi) \in \Lambda$.
\end{cor}
\begin{proof}
It is sufficient to check \eqref{eq:lin21}.
Fix any $(\Psi,\Phi) \in X$ such that $\|(\Psi,\Phi)\|_X \le 1$. By \eqref{eq:lin16},
\begin{equation}\label{eq:lin22}
c_{i\ell} = \mu^2 \sum_{j=1}^k \sum_{m=0}^N (\delta_{ij}\delta_{m\ell} d_i^2 + o(1))
\frac{\int_{\Omega} \(pV_j^{p-1} \Phi_{jm} \Phi + qU_j^{q-1} \Psi_{jm} \Psi\)}
{\int_{\R^N} \left\{p V_{1,0}^{p-1} (\Phi_{1,0}^m)^2 + q U_{1,0}^{q-1} (\Psi_{1,0}^m)^2\right\}}.
\end{equation}
Applying H\"older's inequality and \eqref{eq:Sob}, we have
\begin{equation}\label{eq:lin23}
\begin{aligned}
\left| \int_{\Omega} \(pV_j^{p-1} \Phi_{jm} \Phi + qU_j^{q-1} \Psi_{jm} \Psi\) \right|
&\le C\mu^{-1} \(\|\Phi\|_{L^{p+1}(\Omega)} + \|\Psi\|_{L^{q+1}(\Omega)}\) \\
&\le C\mu^{-1} \|(\Psi, \Phi)\|_X = O(\mu^{-1}).
\end{aligned}
\end{equation}
Using also \eqref{eq:PsPh}, we observe
\begin{equation}\label{eq:lin24}
\begin{aligned}
\|\(P\Psi_{i\ell}, P\Phi_{i\ell}\)\|_X &\le C \(\|V_i^{p-1}\Phi_{i\ell}\|_{L^{p+1 \over p}(\Omega)} + \|U_i^{q-1}\Psi_{i\ell}\|_{L^{q+1 \over q}(\Omega)}\) \\
&\le C \(\|\Phi_{i\ell}\|_{L^{p+1}(\Omega)} + \|\Psi_{i\ell}\|_{L^{q+1}(\Omega)}\) = O(\mu^{-1}).
\end{aligned}
\end{equation}
Putting \eqref{eq:lin20} and \eqref{eq:lin22}-\eqref{eq:lin24} together, we deduce \eqref{eq:lin21}.
\end{proof}

\noindent Lastly, we set the linear operator $\Pi^{\perp}_\dxi: X \to Z_\dxi$ by $\Pi^{\perp}_\dxi = \text{Id}_X - \Pi_\dxi$.

\subsection{Approximate solutions and their errors}\label{subsec:app}
We will solve \eqref{eq:LEs}, or equivalently, \eqref{eq:LEs*} by finding a pair of parameters $(\mbd, \mbxi) \in \Lambda$
and that of functions $(\Psi_\dxi, \Phi_\dxi) \in Z_\dxi$ such that
\begin{align}
&\begin{medsize}
\displaystyle \Pi^{\perp}_\dxi\left[ \(\mbP U_\dxi + \Psi_\dxi,\, \sum_{i=1}^k PV_i + \Phi_\dxi\)
\right. \end{medsize} \label{eq:aux1}
\\
&\qquad \begin{medsize}
\displaystyle - \mci^*\( \left|\sum_{i=1}^k PV_i + \Phi_\dxi\right|^{p-1} \(\sum_{i=1}^k PV_i + \Phi_\dxi\)
+ \ep \left\{\alpha \(\mbP U_\dxi + \Psi_\dxi\) + \beta_1 \(\sum_{i=1}^k PV_i + \Phi_\dxi\) \right\},
\right. \end{medsize} \nonumber \\
&\qquad \qquad \begin{medsize}
\displaystyle \left. \left. \left|\mbP U_\dxi + \Psi_\dxi\right|^{q-1} \(\mbP U_\dxi + \Psi_\dxi\)
+ \ep \left\{\beta_2 \(\mbP U_\dxi + \Psi_\dxi\) + \alpha \(\sum_{i=1}^k PV_i + \Phi_\dxi\) \right\}\) \right] = 0
\end{medsize} \nonumber
\end{align}
and
\begin{align}
&\begin{medsize}
\displaystyle \Pi_\dxi\left[ \(\mbP U_\dxi + \Psi_\dxi,\, \sum_{i=1}^k PV_i + \Phi_\dxi\)
\right. \end{medsize} \label{eq:aux2} \\
&\qquad \begin{medsize}
\displaystyle - \mci^*\( \left|\sum_{i=1}^k PV_i + \Phi_\dxi\right|^{p-1} \(\sum_{i=1}^k PV_i + \Phi_\dxi\)
+ \ep \left\{\alpha \(\mbP U_\dxi + \Psi_\dxi\) + \beta_1 \(\sum_{i=1}^k PV_i + \Phi_\dxi\) \right\},
\right. \end{medsize} \nonumber \\
&\qquad \qquad \begin{medsize}
\displaystyle \left. \left. \left|\mbP U_\dxi + \Psi_\dxi\right|^{q-1} \(\mbP U_\dxi + \Psi_\dxi\)
+ \ep \left\{\beta_2 \(\mbP U_\dxi + \Psi_\dxi\) + \alpha \(\sum_{i=1}^k PV_i + \Phi_\dxi\) \right\}\) \right] = 0 \nonumber
\end{medsize}
\end{align}

Define
\begin{multline}\label{eq:error}
\begin{medsize}
\displaystyle \mce_\dxi = \Pi^{\perp}_\dxi\left[ \(0, \sum_{i=1}^k PV_i\) - \mci^*\(\ep \(\alpha \mbP U_\dxi + \beta_1 \sum_{i=1}^k PV_i\), \right. \right.
\end{medsize} \\
\begin{medsize}
\displaystyle \left. \left. \(\mbP U_\dxi\)^q + \ep \(\beta_2 \mbP U_\dxi + \alpha \sum_{i=1}^k PV_i\)\) \right]
\end{medsize}
\end{multline}
which is the error of a pair $\(\mbP U_\dxi,\, \sum_{i=1}^k PV_i\)$ as an approximate solution to \eqref{eq:LEs}.
\begin{lemma}\label{lemma:error}
For $\ep \in (0,\ep_0)$, there exists a universal constant $C > 0$ such that
\begin{equation}\label{eq:EX}
\|\mce_\dxi\|_X \le C\left[ \mu^{Npq \over q+1} + \mu^{N(p+1) \over q+1}
+ \ep \left\{ \mu^2 \alpha + \mu^{N(p-1) \over p+1} \beta_1 + \(\mu^{Np \over q+1} + \mu^{N(q-1) \over q+1}\) \beta_2 \right\} \right].
\end{equation}
\end{lemma}
\begin{proof}
We decompose $\mce_\dxi$ into
\begin{equation}\label{eq:E}
\begin{aligned}
\mce_\dxi &= \Pi^{\perp}_\dxi\left[\(0,\sum_{i=1}^k PV_i\) - \mci^*\(0, \(\mbP U_\dxi\)^q\)\right] \\
&\ - \ep \Pi^{\perp}_\dxi\left[\mci^*\(\alpha \mbP U_\dxi + \beta_1 \sum_{i=1}^k PV_i,\,
\beta_2 \mbP U_\dxi + \alpha \sum_{i=1}^k PV_i\) \right] =: \mce_1 + \mce_2.
\end{aligned}
\end{equation}

We first calculate $\mce_1$. By Corollary \ref{cor:lin2} and the mapping property of the operator $\mci^*$,
\[\|\mce_1\|_X \le C \left\| \mci^*\(0, \sum_{i=1}^k U_i^q - \(\mbP U_\dxi\)^q\) \right\|_X
\le C \left\| \sum_{i=1}^k U_i^q - \(\mbP U_\dxi\)^q \right\|_{L^{q+1 \over q}(\Omega)}.\]
Fix $l = 1, \cdots, k$ and choose a number $\rho \in (0, \delta_2/3)$. We observe that
\begin{align*}
&\ \left\| \sum_{i=1}^k U_i^q - \(\mbP U_\dxi\)^q \right\|_{L^{q+1 \over q}(B^N(\xi_l,\rho))}^{q+1 \over q} \\
&\le C \left[\sum_{i, l = 1 \atop i \ne l}^k \int_{B^N(\xi_l,\rho)} \left\{ \(U_l^{q-1} U_i\)^{q+1 \over q} + U_i^{q+1} \right\}
+ \mu^{Np \over q} \int_{B^N(\xi_l,\rho)} U_l^{q^2-1 \over q} + \mu^{Np} \right] \\
&\le C \( \mu^{Np \over q} \int_{B^N(\xi_l,\rho)} U_l^{q^2-1 \over q} + \mu^{Np} \) \\
&\le C \(\mu^{Np} + \mu^{N(p+1) \over q}\).
\end{align*}
Here, the first inequality is due to \eqref{eq:bPU} and Lemma \ref{lemma:wth}.
Also, the second inequality follows from \eqref{eq:dec1} and the identities
\begin{equation}\label{eq:id}
(N-2)p-2 = (N-2)(p+1)-N = \frac{N(p+1)}{q+1}.
\end{equation}
The third inequality is justified by
\[\mu^{Np \over q} \int_{B^N(\xi_l,\rho)} U_l^{q^2-1 \over q} \le C \mu^{N(p+1) \over q} \int_{B^N(0,\rho\mu_l^{-1})} U_{1,0}^{q^2-1 \over q} \le C \(\mu^{Np} + \mu^{N(p+1) \over q}\).\]
On the other hand,
\[\left\| \sum_{i=1}^k U_i^q - \(\mbP U_\dxi\)^q \right\|_{L^{q+1 \over q}(\Omega \setminus \cup_{l=1}^k B^N(\xi_l,\rho))}^{q+1 \over q}
\le C \( \sum_{i=1}^k \int_{\Omega \setminus \cup_{l=1}^k B^N(\xi_l,\rho)} U_i^{q+1} + \mu^{Np} \) \le C \mu^{Np}.\]
Thus
\begin{align}
\|\mce_1\|_X &\le \sum_{l=1}^k \left\| \sum_{i=1}^k U_i^q - \(\mbP U_\dxi\)^q \right\|_{L^{q+1 \over q}(B^N(\xi_l,\rho))}
+ \left\| \sum_{i=1}^k U_i^q - \(\mbP U_\dxi\)^q \right\|_{L^{q+1 \over q}(\Omega \setminus \cup_{l=1}^k B^N(\xi_l,\rho))} \nonumber \\
&\le C \(\mu^{Npq \over q+1} + \mu^{N(p+1) \over q+1}\). \label{eq:E1}
\end{align}

Next, we examine $\mce_2$. A straightforward computation shows that \begin{equation}\label{eq:E2}
\begin{aligned}
\|\mce_2\|_X 
&\le C\ep \left[ \left\| \alpha \mbP U_\dxi + \beta_1 \sum_{i=1}^k PV_i \right\|_{L^{p+1 \over p}(\Omega)}
+ \left\| \beta_2 \mbP U_\dxi + \alpha \sum_{i=1}^k PV_i \right\|_{L^{q+1 \over q}(\Omega)} \right] \\
&\le C\ep \left[ \mu^2 \alpha + \mu^{N(p-1) \over p+1} \beta_1 + \(\mu^{Np \over q+1} + \mu^{N(q-1) \over q+1}\) \beta_2 \right].
\end{aligned}
\end{equation}

As a consequence, \eqref{eq:EX} follows from \eqref{eq:E}, \eqref{eq:E1} and \eqref{eq:E2}.
\end{proof}

\section{Analysis on linear problems and its applications} \label{sec:lin}
\subsection{Analysis on linear problems}
For fixed parameters $\alpha,\, \beta_1,\, \beta_2 \in \R$, $\ep > 0$ and $(\dxi) \in \Lambda$, we define a bounded linear operator $L_\dxi: Z_\dxi \to Z_\dxi$ by
\begin{equation}\label{eq:Ldxi}
\begin{medsize}
\displaystyle L_\dxi(\Psi,\Phi) = (\Psi,\Phi) - \Pi_\dxi^{\perp} \mci^*\(p\(\sum_{i=1}^k PV_i\)^{p-1} \Phi + \ep \(\alpha \Psi + \beta_1 \Phi\),
q(\mbP U_\dxi)^{q-1} \Psi + \ep \(\beta_2 \Psi + \alpha \Phi\)\).
\end{medsize}
\end{equation}
Refer to Corollary \ref{cor:lin2}.
\begin{rmk}\label{rmk:lin}
The operator $L_\dxi$ can be rewritten as $L_\dxi = \text{Id} + \mck$ on $Z_\dxi$ where $\text{Id}$ is the identity operator and $\mck$ is a compact operator. Let us examine this fact.

Clearly, the first term of $L_\dxi(\Psi,\Phi)$ in \eqref{eq:Ldxi} equals $\text{Id}(\Psi,\Phi)$.

The second term needs to be treated more carefully.
Because we are assuming that $p > 1$, two functions $\(\sum_{i=1}^k PV_i\)^{p-1}$ and $(\mbP U_\dxi)^{q-1}$ are contained in $L^{\infty}(\Omega)$.
Also, we can take a number $\zeta > 0$ so small that $p+1-\zeta \ge \frac{p+1}{p}$ and $q+1-\zeta \ge \frac{q+1}{q}$.
Hence, the Rellich-Kondrachov theorem implies that the map
\begin{multline*}
\begin{medsize}
\displaystyle (\Psi,\Phi) \in Z_\dxi \mapsto \(p\(\sum_{i=1}^k PV_i\)^{p-1} \Phi + \ep \(\alpha \Psi + \beta_1 \Phi\), q(\mbP U_\dxi)^{q-1} \Psi + \ep \(\beta_2 \Psi + \alpha \Phi\)\)
\end{medsize} \\
\in L^{p+1-\zeta}(\Omega) \times L^{\min\{q+1-\zeta,2\}}(\Omega) \subset L^{p+1 \over p}(\Omega) \times L^{q+1 \over q}(\Omega) = \text{Dom}(\mci^*)
\end{multline*}
is compact. If we set $\mck$ as the composition of the above map and bounded operators $\mci^*$ and $\Pi_\dxi^{\perp}$,
it is also an compact operator and the second term of $L_\dxi(\Psi,\Phi)$ in \eqref{eq:Ldxi} equals $\mck(\Psi,\Phi)$.
\end{rmk}

The main result of this subsection is the following proposition, which concerns invertibility of the operator $L_\dxi$ on $Z_\dxi$.
\begin{prop}\label{prop:lin}
Reduce the value of $\ep_0 > 0$ if necessary. Then there is a universal constant $C > 0$ such that for each $\ep \in (0,\ep_0)$ and $(\dxi) \in \Lambda$, the operator $L_\dxi$ satisfies
\begin{equation}\label{eq:lin31}
\|L_\dxi(\Psi,\Phi)\|_X \ge C\|(\Psi,\Phi)\|_X \quad \text{for every } (\Psi,\Phi) \in Z_\dxi.
\end{equation}
\end{prop}
\noindent Once Proposition \ref{prop:lin} is established, Remark \ref{rmk:lin} and the Fredholm alternative will readily imply the following assertion.
\begin{cor}\label{cor:lin}
For each $\ep \in (0,\ep_0)$, $(\dxi) \in \Lambda$ and $(H_1, H_2) \in Z_\dxi$,
there exists a unique solution $(\Psi, \Phi) \in Z_\dxi$ to the linear problem $L_\dxi(\Psi, \Phi) = (H_1, H_2)$. Furthermore,
\[\|(H_1,H_2)\|_X \ge C\|(\Psi,\Phi)\|_X\]
where $C > 0$ is the universal constant in Proposition \ref{prop:lin}.
\end{cor}

The rest of this subsection is devoted to the proof of Proposition \ref{prop:lin}.

Suppose that \eqref{eq:lin31} does not hold. There exist sequences $\{\ep_n\}_{n \in \N}$ of positive small numbers,
\[\{(\dxin) = (d_{1,n}, \cdots, d_{k,n}, \xi_{1,n}, \cdots, \xi_{k,n})\}_{n \in \N} \subset \Lambda, \quad \{(\Psi_n,\Phi_n) \in Z_\dxin\}_{n \in \N}\]
and
\begin{equation}\label{eq:lin311}
\{(H_{1n},H_{2n}) := L_\dxin(\Psi_n,\Phi_n)\}_{n \in \N}
\end{equation}
such that $\ep_n \to 0$, $(\dxin) \to (\dxii) \in \Lambda$ as $n \to \infty$,
\begin{equation}\label{eq:lin32}
\|(\Psi_n,\Phi_n)\|_X = 1 \quad \text{for all } n \in \N
\quad \text{and} \quad
\|(H_{1n},H_{2n})\|_X \to 0 \quad \text{as } n \to \infty.
\end{equation}
For convenience's sake, let $\{\mu_n\}_{n \in \N}$ be the sequence such that each $\mu_n$ is the positive small number corresponding to $\ep_n$.
We also set $\mu_{i,n} = \mu_n d_{i,n}$,
\[PV_{i,n} = PV_{\mu_{i,n}, \xi_{i,n}}, \quad
\mbP U_n = \mbP U_\dxin \quad \text{and} \quad
(P\Psi_{i\ell,n}, P\Phi_{i\ell,n}) = (\Psi_{\mu_{i,n} ,\xi_{i,n}}^{\ell}, \Phi_{\mu_{i,n}, \xi_{i,n}}^{\ell}).\]

According to \eqref{eq:Ldxi} and \eqref{eq:lin311}, there is a sequence $\{c_{i\ell,n}: i = 1, \cdots, k \text{ and } \ell = 0, \cdots, N\}_{n \in \N}$ of coefficients such that
\begin{multline}\label{eq:lin33}
\begin{medsize}
\displaystyle (\Psi_n,\Phi_n) - \mci^*\(p\(\sum_{i=1}^k PV_{i,n}\)^{p-1} \Phi_n + \ep_n \(\alpha \Psi_n + \beta_1 \Phi_n\),
q(\mbP U_n)^{q-1} \Psi_n + \ep_n \(\beta_2 \Psi_n + \alpha \Phi_n\)\)
\end{medsize} \\
= (H_{1n}, H_{2n}) + \sum_{i=1}^k \sum_{\ell=0}^N c_{i\ell,n} \(P\Psi_{i\ell,n}, P\Phi_{i\ell,n}\),
\end{multline}
which also reads
\begin{equation}\label{eq:lin331}
\begin{cases}
\begin{medsize}
\displaystyle -\Delta \Psi_n = p\(\sum_{i=1}^k PV_{i,n}\)^{p-1} \Phi_n + \ep_n \(\alpha \Psi_n + \beta_1 \Phi_n\)
-\Delta H_{1n} + \sum_{i=1}^k \sum_{\ell=0}^N c_{i\ell,n} \(p V_{i,n}^{p-1} \Phi_{i\ell,n}\)
\end{medsize} &\text{in } \Omega, \\
\begin{medsize}
\displaystyle - \Delta \Phi_n = q(\mbP U_n)^{q-1} \Psi_n + \ep_n \(\beta_2 \Psi_n + \alpha \Phi_n\)
-\Delta H_{2n} + \sum_{i=1}^k \sum_{\ell=0}^N c_{i\ell,n} \(q U_{i,n}^{q-1} \Psi_{i\ell,n}\)
\end{medsize} &\text{in } \Omega, \\
\Psi_n = \Phi_n = 0 &\text{on } \pa\Omega.
\end{cases}
\end{equation}

\begin{lemma}
It holds that
\begin{equation}\label{eq:lin41}
\mu_n^{-1} \sum_{i=1}^k \sum_{\ell=0}^N \left| c_{i\ell,n} \right| \to 0 \quad \text{as } n \to \infty.
\end{equation}
\end{lemma}
\begin{proof}
Fix $j = 1, \cdots, k$ and $m = 0, \cdots, N$.
By testing \eqref{eq:lin33} with $(P\Phi_{jm,n}, P\Psi_{jm,n}) \in L^{p+1}(\Omega) \times L^{q+1}(\Omega)$ and using \eqref{eq:PsPh}, we obtain
\begin{equation}\label{eq:lin42}
\begin{aligned}
&\begin{medsize}
\ \displaystyle \int_{\Omega} \left[ p \left\{V_{j,n}^{p-1} - \(\sum_{i=1}^k PV_{i,n}\)^{p-1} \right\} P\Phi_{jm,n} \Phi_n
+ q \left\{U_{j,n}^{q-1} - (\mbP U_n)^{q-1} \right\} P\Psi_{jm,n} \Psi_n \right]
\end{medsize} \\
&\ - \int_{\Omega} \ep_n \left[ \(\alpha \Psi_n + \beta_1 \Phi_n\) P\Phi_{jm,n}
+ \(\beta_2 \Psi_n + \alpha \Phi_n\) P\Psi_{jm,n} \right] \\
&= \int_{\Omega} \left[ (-\Delta H_{1n}) P\Phi_{jm,n} + (-\Delta H_{2n}) P\Psi_{jm,n} \right] \\
&\ + \int_{\Omega} \sum_{i=1}^k \sum_{\ell=0}^N c_{i\ell,n} \(p V_{i,n}^{p-1} \Phi_{i\ell,n} \cdot P\Phi_{jm,n} + q U_{i,n}^{q-1} \Psi_{i\ell,n} \cdot P\Psi_{jm,n} \).
\end{aligned}
\end{equation}
Let $J_{L1}$ and $J_{L2}$ denote the first and second integral on the left-hand side of \eqref{eq:lin42}, respectively.
Also, let $J_{R1}$ and $J_{R2}$ denote those on the right-hand side of \eqref{eq:lin42}, respectively.
We will derive \eqref{eq:lin41} by estimating each integral $J_{L1}$, $J_{L2}$, $J_{R1}$ and $J_{R2}$.

First, we have
\begin{align*}
&\ \begin{medsize}
\displaystyle \left| \int_{\Omega} \left\{V_{j,n}^{p-1} - \(\sum_{i=1}^k PV_{i,n}\)^{p-1} \right\} P\Phi_{jm,n} \Phi_n \right|
\end{medsize} \\
&\le \begin{medsize}
\displaystyle C \left\| \left\{ V_{j,n}^{p-1} - \(\sum_{i=1}^k PV_{i,n}\)^{p-1} \right\} P\Phi_{jm,n} \right\|_{L^{p+1 \over p}(\Omega)}
\end{medsize} \\
&\le \begin{medsize}
\displaystyle C \sum_{l=1}^k \left\| \left\{ V_{j,n}^{p-1}
- \(\sum_{i=1}^k PV_{i,n}\)^{p-1} \right\} P\Phi_{jm,n} \right\|_{L^{p+1 \over p}(B^N(\xi_{l,n},\rho))} + O(\mu_n^{-1} \mu_n^{((N-2)p-2){p \over p+1}}).
\end{medsize}
\end{align*}
If $l \ne j$,
\begin{align*}
&\ \begin{medsize}
\displaystyle \left\| \left\{ V_{j,n}^{p-1}
- \(\sum_{i=1}^k PV_{i,n}\)^{p-1} \right\} P\Phi_{jm,n} \right\|_{L^{p+1 \over p}(B^N(\xi_{l,n},\rho))}
\end{medsize} \\
&\le C \mu_n^{{N \over q+1}-1} \(\int_{B^N(\xi_{l,n},\rho)} V_{l,n}^{p^2-1 \over p}\)^{p \over p+1} + O(\mu_n^{-1} \mu_n^{((N-2)p-2){p \over p+1}}) \\
&\le C\mu_n^{-1} \left[ \mu_n^{N-2} + \mu_n^{((N-2)p-2){p \over p+1}} \right] = o(\mu_n^{-1}).
\end{align*}
If $l = j$,
\[\begin{medsize}
\displaystyle \left\| \left\{ V_{j,n}^{p-1}
- \(\sum_{i=1}^k PV_{i,n}\)^{p-1} \right\} P\Phi_{jm,n} \right\|_{L^{p+1 \over p}(B^N(\xi_{l,n},\rho))}
\end{medsize}
\le C\mu_n^{-1} \mu_n^{N(p-1) \over q+1}  = o(\mu_n^{-1}).\]
Here, we used the relation that $p-1 \in (0,1)$. Combining the computations, we arrive at
\[\begin{medsize}
\displaystyle \left| \int_{\Omega} \left\{V_{j,n}^{p-1} - \(\sum_{i=1}^k PV_{i,n}\)^{p-1} \right\} P\Phi_{jm,n} \Phi_n \right|
\end{medsize}
= o(\mu_n^{-1}).\]
An analogous calculation yields
\[\left| \int_{\Omega} \left\{U_{j,n}^{q-1} - (\mbP U_n)^{q-1} \right\} P\Psi_{jm,n} \Psi_n \right| \le C\mu_n^{-1} \(\mu_n^{(N-2)p-2} 
+ \mu_n^{{Np(q-1) \over q+1}}\) = o(\mu_n^{-1}).\]
Therefore,
\begin{equation}\label{eq:lin43}
J_{L1} = o(\mu_n^{-1}).
\end{equation}

Second, using the assumption that $p > 1$, we see
\begin{equation}\label{eq:lin44}
\begin{aligned}
|J_{L2}| &\le C \ep_n \mu_n^{-1} \(\|\Psi_n\|_{L^{p+1 \over p}(\Omega)} + \|\Phi_n\|_{L^{p+1 \over p}(\Omega)}
+ \|\Psi_n\|_{L^{q+1 \over q}(\Omega)} + \|\Phi_n\|_{L^{q+1 \over q}(\Omega)}\) \\
&\le C \ep_n \mu_n^{-1} \|(\Psi_n,\Phi_n)\|_X = o(\mu_n^{-1}).
\end{aligned}
\end{equation}

Third, taking into account \eqref{eq:PsPh}, we obtain
\begin{equation}\label{eq:lin45}
\begin{aligned}
|J_{R1}| &\le \int_{\Omega} \left| pV_{i,n}^{p-1}\Phi_{i\ell,n} H_{2n} + qU_{i,n}^{q-1}\Psi_{i\ell,n} H_{1n} \right| \\
&\le p\|V_{i,n}\|^{p-1}_{L^{p+1}(\Omega)} \|\Phi_{i\ell,n}\|_{L^{p+1}(\Omega)} \|H_{2n}\|_{L^{p+1}(\Omega)} \\
&\hspace{150pt} + q\|U_{i,n}\|^{q-1}_{L^{q+1}(\Omega)} \|\Psi_{i\ell,n}\|_{L^{q+1}(\Omega)} \|H_{1n}\|_{L^{q+1}(\Omega)} \\
&\le C\mu_n^{-1} \|(H_{1n},H_{2n})\|_X = o(\mu_n^{-1}).
\end{aligned}
\end{equation}

Lastly, the proof of Lemma \ref{lemma:lin1} shows
\begin{equation}\label{eq:lin46}
J_{R2} = \mu_n^{-2} \sum_{i=1}^k \sum_{\ell=0}^N c_{i\ell,n}
\left[\delta_{ij} \delta_{\ell m} d_j^{-2} \int_{\R^N} \left\{p V_{1,0}^{p-1} (\Phi_{1,0}^m)^2 + q U_{1,0}^{q-1} (\Psi_{1,0}^m)^2\right\} + o(1) \right].
\end{equation}

By putting \eqref{eq:lin43}-\eqref{eq:lin46} into \eqref{eq:lin42}, we obtain \eqref{eq:lin41} as claimed.
\end{proof}

\begin{lemma}
Passing to a subsequence, it holds that
\begin{equation}\label{eq:lin51}
\left\| \(\sum_{i=1}^k PV_{i,n}\)^{p-1} \Phi_n \right\|_{L^{p+1 \over p}(\Omega)}
+ \left\| (\mbP U_n)^{q-1} \Psi_n \right\|_{L^{q+1 \over q}(\Omega)} \to 0 \quad \text{as } n \to \infty.
\end{equation}
\end{lemma}
\begin{proof}
Fixing any $l = 1, \cdots, k$, let $\chi: \R^N \to [0,1]$ be a smooth cut-off function such that
\[\chi(x) = \begin{cases}
1 &\text{in } B^N(\xi_{l,\infty},\rho),\\
0 &\text{in } \Omega \setminus B^N(\xi_{l,\infty},2\rho),
\end{cases}
\quad |\nabla \chi(x)| \le \frac{2}{\rho}
\quad \text{and} \quad |\nabla^2 \chi(x)| \le \frac{4}{\rho^2}.\]
If we set
\[\(\wtps_n(y), \wtph_n(y)\) = \(\mu_{l,n}^{N \over q+1} (\chi \Psi_n)(\mu_{l,n} y + \xi_{l,n}),
\mu_{l,n}^{N \over p+1} (\chi \Phi_n)(\mu_{l,n} y + \xi_{l,n})\) \quad \text{for } y \in \R^N,\]
a straightforward computation shows
\begin{equation}\label{eq:lin512}
\|\Delta \wtps_n(y)\|_{L^{\frac{p+1}{p}}(\R^N)} + \|\Delta \wtph_n(y)\|_{L^{\frac{q+1}{q}}(\R^N)} \le C \quad \text{for all } n \in \N.
\end{equation}
Therefore, for any small number $\zeta > 0$,
\begin{equation}\label{eq:lin52}
\(\wtps_n, \wtph_n\) \rightharpoonup \(\wtps_{\infty}, \wtph_{\infty}\)
\quad \begin{cases}
\text{weakly in } \dot{W}^{2,\frac{p+1}{p}}(\R^N) \times \dot{W}^{2,\frac{q+1}{q}}(\R^N), \\
\text{strongly in } L^{q+1-\zeta}_{\text{loc}}(\R^N) \times L^{p+1-\zeta}_{\text{loc}}(\R^N), \\
\text{almost everywhere in } \R^N,
\end{cases}
\end{equation}
along a subsequence. Furthermore, using
\[\Delta \wtps_n(y) = \mu_{l,n}^{pN \over p+1} \(\chi \Delta \Psi_n + 2 \nabla \chi \cdot \nabla \Psi_n + \Psi_n \Delta \chi\)(\mu_{l,n}y+\xi_{l,n}) \quad \text{for } y \in \R^N,\]
an analogous formula for $\Delta \wtph_n(y)$ and \eqref{eq:lin331},
we obtain a system of equations satisfied by $(\wtps_n, \wtph_n)$.
Taking $n \to \infty$ in it and employing \eqref{eq:lin32} and \eqref{eq:lin41},
we see that $(\wtps_{\infty}, \wtph_{\infty})$ satisfies \eqref{eq:lin-sys}.
Indeed, by the dominated convergence theorem,
\[\lim_{n \to \infty} \mu_{l,n}^{pN \over p+1} \int_{\text{supp}\, \Theta} \left[p\(\sum_{i=1}^k PV_{i,n}\)^{p-1} \Phi_n \right](\mu_{l,n}y+\xi_{l,n}) \Theta(y) dy
= \int_{\text{supp}\, \Theta} pV_{1,0}^{p-1} \wtph_{\infty} \Theta\]
for each $\Theta \in C^{\infty}_c(\R^N)$, which corresponds to the right-hand side of the first equation in \eqref{eq:lin-sys}.
The other terms in the system of $(\wtps_n, \wtph_n)$ can be treated similarly.
On the other hand, the condition $(\Psi_n,\Phi_n) \in Z_\dxin$ asserts
\begin{align*}
&\ \int_{\R^N} \(p V_{1,0}^{p-1} \Phi_{1,0}^{\ell} \wtph_{\infty} + q U_{1,0}^{q-1} \Psi_{1,0}^{\ell} \wtps_{\infty}\) \\
&= \lim_{n \to \infty} \int_{\R^N} \(p V_{1,0}^{p-1} \Phi_{1,0}^{\ell} \wtph_n + q U_{1,0}^{q-1} \Psi_{1,0}^{\ell} \wtps_n\) \\
&= \lim_{n \to \infty} \int_{B^N(0,3\rho\mu_{l,n}^{-1}) \setminus B^N(0,\rho\mu_{l,n}^{-1}/2)} \left[ p V_{1,0}^{p-1} \Phi_{1,0}^{\ell} \cdot \mu_{l,n}^{N \over p+1} \{(\chi-1) \Phi_n\}(\mu_{l,n}y+\xi_{l,n}) \right. \\
&\hspace{160pt} \left. + q U_{1,0}^{q-1} \Psi_{1,0}^{\ell} \cdot \mu_{l,n}^{N \over q+1} \{(\chi-1) \Psi_n\}(\mu_{l,n}y+\xi_{l,n}) \right] dy \\
&= \lim_{n \to \infty} O\(\mu_n^{((N-2)p-2){p \over p+1}} + \mu_n^{Npq \over q+1}\) = 0
\end{align*}
for all $\ell = 0, \cdots, N$. In view of Lemma \ref{lemma:nondeg},
\begin{equation}\label{eq:lin53}
\(\wtps_{\infty}, \wtph_{\infty}\) = (0,0).
\end{equation}

We now show that the $L^{p+1 \over p}(\Omega)$-norm of $\(\sum_{i=1}^k PV_{i,n}\)^{p-1} \Phi_n$ converges to 0 as $n \to \infty$.
Fix any $\frac{2}{N} < \kappa_2 < 1$. From \eqref{eq:lin512}, \eqref{eq:lin52} with a selection $p+1-\zeta > \frac{p+1}{p}$ (possible for $p > 1$) and \eqref{eq:lin53}, we observe
\begin{equation}\label{eq:lin54}
\begin{aligned}
&\ \left\| \(\sum_{i=1}^k PV_{i,n}\)^{p-1} \Phi_n \right\|_{L^{p+1 \over p}(B^N(\xi_{l,n},\mu_{l,n}^{\kappa_2}))} \\
&= \left\| \(V_{1,0} + O(\mu_n^{N-2})\)^{p-1} \wtph_n \right\|_{L^{p+1 \over p}(B^N(0,\mu_{l,n}^{\kappa_2-1}))} \\
&\le C \(\left\| \wtph_n \right\|_{L^{p+1 \over p}(B^N(0,r))}
+ \left\| V_{1,0} + O(\mu_n^{N-2}) \right\|_{L^{p+1}(B^N(0,\mu_{l,n}^{\kappa_2-1}) \setminus B^N(0,r))}^{p-1} \)
\end{aligned}
\end{equation}
for each $r > 0$. For arbitrarily given $\vep > 0$, we can choose suitable $n_0 \in \N$ and $r > 0$ large so that the rightmost side of \eqref{eq:lin54} is bounded by $\vep$ for all $n \ge n_0$. Accordingly,
\begin{equation}\label{eq:lin55}
\left\| \(\sum_{i=1}^k PV_{i,n}\)^{p-1} \Phi_n \right\|_{L^{p+1 \over p}(B^N(\xi_{l,n},\mu_{l,n}^{\kappa_2}))} = o(1).
\end{equation}
Also, by applying H\"older's inequality, \eqref{eq:dec3} and \eqref{eq:PV}, we easily get
\begin{equation}\label{eq:lin511}
\left\| \(\sum_{i=1}^k PV_{i,n}\)^{p-1} \Phi_n \right\|_{L^{p+1 \over p}(\Omega \setminus \cup_{l=1}^k B^N(\xi_{l,n},\mu_{l,n}^{\kappa_2}))} \le C \mu_n^{(1-\kappa_2)((N-2)p-2) \frac{p-1}{p+1}} = o(1).
\end{equation}
The assertion follows from \eqref{eq:lin55} and \eqref{eq:lin511}.

\medskip
The claim for $(\mbP U_n)^{q-1} \Psi_n$ can be proved in a similar fashion. The proof is completed.
\end{proof}

\begin{proof}[Completion of Proof of Proposition \ref{prop:lin}]
By \eqref{eq:lin32}, \eqref{eq:lin331}, \eqref{eq:lin41} and \eqref{eq:lin51},
\begin{align*}
1 &= \|(\Psi_n,\Phi_n)\|_X \\ 
&\le C \(\left\|\(\sum_{i=1}^k PV_{i,n}\)^{p-1} \Phi_n \right\|_{L^{\frac{p+1}{p}}(\Omega)}
+ \left\| (\mbP U_n)^{q-1} \Psi_n \right\|_{L^{\frac{q+1}{q}}(\Omega)} + \ep_n \|\alpha \Psi_n + \beta_1 \Phi_n\|_{L^{\frac{p+1}{p}}(\Omega)} \right. \\
&\hspace{30pt} \left. + \ep_n \|\beta_2 \Psi_n + \alpha \Phi_n\|_{L^{\frac{q+1}{q}}(\Omega)}
+ \|(H_{1n},H_{2n})\|_X + \mu_n^{-1} \sum_{i=1}^k \sum_{\ell=0}^N |c_{i\ell,n}|\) \\
&\to 0 \quad \text{as } n \to \infty,
\end{align*}
up to a subsequence. This is a contradiction, and so \eqref{eq:lin31} must be true.
\end{proof}

\subsection{Nonlinear problems}
Using Corollary \ref{cor:lin}, we solve the auxiliary equation \eqref{eq:aux1}.
Recall that $\mce_\dxi \in Z_\dxi$ is the error term introduced in \eqref{eq:error}, whose $X$-norm was estimated in \eqref{eq:EX}.
\begin{prop}\label{prop:nonlin}
Take smaller $\ep_0 > 0$ if necessary. Then for each $\ep \in (0,\ep_0)$ and $(\dxi) \in \Lambda$,
one has the unique solution $(\Psi_\dxi^\ep, \Phi_\dxi^\ep) \in Z_\dxi$ to \eqref{eq:aux1} satisfying
\begin{equation}\label{eq:nonlin}
\left\| (\Psi_\dxi^\ep, \Phi_\dxi^\ep) \right\|_X \le C\|\mce_\dxi\|_X
\end{equation}
for some universal constant $C > 0$.
Furthermore, $(\Psi_\dxi^\ep, \Phi_\dxi^\ep) \in (L^{\infty}(\Omega))^2$
and the map $(\dxi) \in \Lambda \mapsto (\Psi_\dxi^\ep, \Phi_\dxi^\ep) \in X$ is of $C^1$-class.
\end{prop}
\begin{proof}
We reformulate \eqref{eq:aux1} to a fixed point problem
\begin{align*}
(\Psi, \Phi) &= T_\dxi(\Psi, \Phi) \\
&:= \begin{medsize}
\displaystyle L_\dxi^{-1}\left[ \Pi^{\perp}_\dxi \mci^*\( \left|\sum_{i=1}^k PV_i + \Phi\right|^{p-1}\(\sum_{i=1}^k PV_i + \Phi\)
- \(\sum_{i=1}^k PV_i\)^p - p\(\sum_{i=1}^k PV_i\)^{p-1}\Phi, \right. \right.
\end{medsize} \\
&\begin{medsize}
\displaystyle \hspace{100pt} \left. \left. \left|\mbP U_\dxi + \Psi\right|^{q-1}\(\mbP U_\dxi + \Psi\)
- (\mbP U_\dxi)^q - q (\mbP U_\dxi)^{q-1}\Psi\) - \mce_\dxi \right]
\end{medsize}
\end{align*}
where the existence of the bounded inverse $L_\dxi^{-1}$ of the linear operator $L_\dxi$ on $Z_\dxi$ is guaranteed in Corollary \ref{cor:lin}.

A standard argument shows that $T_\dxi$ is a contraction map on
\[\{(\Psi, \Phi) \in Z_\dxi: \|(\Psi,\Phi)\|_X \le C\|\mce_\dxi\|_X \}\]
for some $C > 0$ large enough, of which we omit the details.
See the proof of Proposition 1.8 in \cite{MP} where the corresponding result in the setting of the classical Lane-Emden equation is dealt with.
Therefore, there exists the unique solution $(\Psi_\dxi^\ep, \Phi_\dxi^\ep) \in Z_\dxi$ to \eqref{eq:aux1} satisfying \eqref{eq:nonlin}.

For each fixed $\ep \in (0,\ep_0)$ and $(\dxi) \in \Lambda$, the uniform boundedness of $(\Psi_\dxi^\ep, \Phi_\dxi^\ep)$ on $\ovom$
essentially follows from the HLS inequality.
In Appendix \ref{sec:app-b}, we will deduce it as a consequence of a general regularity result.

Lastly, we infer from the implicit function theorem, the Fredholm alternative and the uniform boundedness of $(\Psi_\dxi^\ep, \Phi_\dxi^\ep)$
that $(\dxi) \mapsto (\Psi_\dxi^\ep, \Phi_\dxi^\ep)$ is a $C^1$-map for small $\ep$.
\end{proof}

\subsection{Finite dimensional reduction}
Define an energy functional $I: X \to \R$ by
\begin{multline}\label{eq:Iep}
I_{\ep}(u,v) = \int_{\Omega} \nabla u \cdot \nabla v - \frac{1}{p+1} \int_{\Omega} |v|^{p+1} - \frac{1}{q+1} \int_{\Omega} |u|^{q+1} \\
- \ep \(\alpha \int_{\Omega} uv + \frac{\beta_1}{2} \int_{\Omega} v^2 + \frac{\beta_2}{2} \int_{\Omega} u^2\) \quad \text{for } (u,v) \in X.
\end{multline}
Since $q \ge p > 1$, $I_{\ep}$ is of class $C^2(X)$. Also, $(u,v) \in X$ is a solution to \eqref{eq:LEs} if and only if it is a positive critical point of $I_{\ep}$.
We set the reduced energy $J_{\ep}: \Lambda \to \R$ as
\begin{equation}\label{eq:Jep}
J_{\ep}(\dxi) = I_{\ep}\(\mbP U_\dxi + \Psi_\dxi,\, \sum_{i=1}^k PV_i + \Phi_\dxi\)
\end{equation}
where $(\Psi_\dxi, \Phi_\dxi)$ is the solution to \eqref{eq:aux1} for a given $(\dxi) \in \Lambda$, which was found in Corollary \ref{cor:lin}.
We also write $\text{int}(\Lambda)$ to denote the interior of $\Lambda$.

\begin{prop}\label{prop:red}
Reduce the value of $\ep_0 > 0$ if necessary.
If $(\dxi) \in \textnormal{int}(\Lambda)$ is a critical point of $J_{\ep}$ for $\ep \in (0,\ep_0)$, then the function
\[(U_\dxi^\ep, V_\dxi^\ep) := \(\mbP U_\dxi + \Psi_\dxi,\, \sum_{i=1}^k PV_i + \Phi_\dxi\) \in X\]
is a critical point of $I_{\ep}$ in $(C^2(\ovom))^2$.
In particular, it is a solution to \eqref{eq:LEs} which blows-up at $k$ points.
Moreover, if $\beta_1, \beta_2 \ge 0$, we may assume that its components are positive.
\end{prop}
\begin{proof}
We first claim that $(U_\dxi^\ep, V_\dxi^\ep)$ is a critical point of $I_{\ep}$ for $\ep \in (0,\ep_0)$
if $(\dxi) \in \textnormal{int}(\Lambda)$ is a critical point of $J_{\ep}$.

Let $s$ be a component of $(\dxi) \in \Lambda$. By \eqref{eq:aux1},
\begin{equation}\label{eq:red1}
\begin{aligned}
0 &= \pa_s J_{\ep}(\dxi) \\
&= \sum_{i=1}^k \sum_{\ell=0}^N c_{i\ell} \int_{\Omega} \left[p V_i^{p-1} \Phi_{i\ell}\, \pa_s \(\sum_{j=1}^k PV_j + \Phi_\dxi\) + q U_i^{q-1} \Psi_{i\ell}\, \pa_s\(\mbP U_\dxi + \Psi_\dxi\) \right].
\end{aligned}
\end{equation}
It is enough to show that all $c_{i\ell}$'s are zero. Hereafter, we fix $i$ and $\ell$.

We compute that
\begin{equation}\label{eq:red2}
\begin{aligned}
&\ \int_{\Omega} \left[p V_i^{p-1} \Phi_{i\ell}\, \pa_s \(\sum_{j=1}^k PV_j\) + q U_i^{q-1} \Psi_{i\ell}\, \pa_s (\mbP U_\dxi) \right] \\
&= \int_{\Omega} \left[p\(\sum_{j=1}^k PV_j\)^{p-1} \(\sum_{j=1}^k \pa_s PV_j\) P\Phi_{i\ell} + \sum_{j=1}^k qU_j^{q-1} \pa_s U_j P\Psi_{i\ell} \right] \\
&= \begin{cases}
\displaystyle - \delta_{il}\delta_{\ell\, 0} \mu^{-1} \left[d_l^{-2} \int_{\R^N} \left\{ pV_{1,0}^{p-1} (\Phi_{1,0}^0)^2 + qU_{1,0}^{q-1} (\Psi_{1,0}^0)^2 \right\} + o(1) \right] &\text{if } s = d_l,\\
\displaystyle - \delta_{il}\delta_{\ell\, m} \mu^{-2} \left[d_l^{-2} \int_{\R^N} \left\{ pV_{1,0}^{p-1} (\Phi_{1,0}^m)^2 + qU_{1,0}^{q-1} (\Psi_{1,0}^m)^2 \right\} + o(1) \right] &\text{if } s = \xi_{lm}.
\end{cases}
\end{aligned}
\end{equation}
Here, the first equality is a consequence of \eqref{eq:PUPV}, \eqref{eq:PsPh} and \eqref{eq:mbP} as well as integration by parts.
The validity of the second equality comes from Lemmas \ref{lemma:PUPV} and \ref{lemma:PsPh}, and the relations
\begin{equation}\label{eq:red21}
\pa_sU_j = \begin{cases}
- \delta_{jl} \mu \Psi_{l0} &\text{if } s = d_l,\\
- \delta_{jl} \Psi_{lm} &\text{if } s = \xi_{lm},
\end{cases}
\quad \text{and so} \quad
\pa_sPV_j = \begin{cases}
- \delta_{jl} \mu P\Phi_{l0} &\text{if } s = d_l,\\
- \delta_{jl} P\Phi_{lm} &\text{if } s = \xi_{lm}
\end{cases}
\end{equation}
where we denoted $\xi_j = (\xi_{j1}, \cdots, \xi_{jN}) \in \Omega$. Refer to the proof of Lemma \ref{lemma:lin1}.

On the other hand, a standard rescaling argument combined with \eqref{eq:dec1} shows
\begin{equation}\label{eq:dec7}
\left| \nabla^\kappa U_{1,0}(x) \right| \le \frac{C}{|x|^{(N-2)p-2+\kappa}}
\quad \text{and} \quad
\left| \nabla^\kappa V_{1,0}(x) \right| \le \frac{C}{|x|^{N-2+\kappa}} \quad \text{for } |x| \ge 1
\end{equation}
for $\kappa = 0, 1, 2$. Moreover,
\begin{multline}\label{eq:red22}
\pa_s \Psi_{i\ell}(x) \\
= \begin{cases}
\begin{medsize}
\displaystyle - \delta_{il} \mu \mu_l^{-{N \over q+1}-2} \left[\sum_{m=1}^N \mu_l^{-1} (x-\xi_l)_m (\pa_{\ell m} U_{1,0})(\mu_l^{-1}(x-\xi_l))
+ \(\frac{N}{q+1}+1\) (\pa_{\ell} U_{1,0})(\mu_l^{-1}(x-\xi_l)) \right]
\end{medsize} \\
\hspace{240pt} \text{if } \ell = 1, \cdots, N \text{ and } s = d_l,\\
\begin{medsize}
\displaystyle - \delta_{il} \mu \mu_l^{-{N \over q+1}-2} \left[\sum_{m=1}^N \mu_l^{-1} (x-\xi_l)_m (\pa_m \Phi_{1,0}^0)(\mu_l^{-1}(x-\xi_l))
+ \(\frac{N}{q+1}+1\) \Phi_{1,0}^0(\mu_l^{-1}(x-\xi_l)) \right]
\end{medsize} \\
\hspace{240pt} \text{if } \ell = 0 \text{ and } s = d_l,\\
- \delta_{il} \mu_l^{-{N \over q+1}-2} (\pa_{\ell m} U_{1,0})(\mu_l^{-1}(x-\xi_l))
\hspace{73pt} \text{if } \ell = 1, \cdots, N \text{ and } s = \xi_{lm},\\
- \delta_{il} \mu_l^{-{N \over q+1}-2} (\pa_m \Phi_{1,0}^0)(\mu_l^{-1}(x-\xi_l))
\hspace{76pt} \text{if } \ell = 0 \text{ and } s = \xi_{lm}
\end{cases}
\end{multline}
for $x \in \Omega$. Using \eqref{eq:red21}, \eqref{eq:red22}, analogous expressions for $\pa_s V_i$ and $\pa_s \Phi_{i\ell}$, \eqref{eq:dec3} and \eqref{eq:dec5}, we find
\begin{equation}\label{eq:red4}
\begin{cases}
\displaystyle \left| \left[\pa_s(V_i^{p-1}\Phi_{i\ell})\right](x) \right| \le \frac{C \delta_{il} \mu_l^{{Np \over q+1}-1}}{\mu_l^{\zeta_1} + |x-\xi_l|^{\zeta_1}},\\
\displaystyle \left| \left[\pa_s(U_i^{q-1}\Psi_{i\ell})\right](x) \right| \le \frac{C \delta_{il} \mu_l^{{Npq \over q+1}-1}}{\mu_l^{\zeta_2} + |x-\xi_l|^{\zeta_2}}
\end{cases}
\quad \text{if } \ell = 0, \dots, N \text{ and } s = d_l \text{ or } \xi_{lm}
\end{equation}
for $x \in \Omega$ where
\begin{equation}\label{eq:red5}
\zeta_1 := \begin{cases}
(N-2)p &\text{if } s = d_l,\\
(N-2)p + 1 &\text{if } s = \xi_{lm}
\end{cases}
\quad \text{and} \quad
\zeta_2 := \begin{cases}
((N-2)p-2)q &\text{if } s = d_l,\\
((N-2)p-2)q + 1 &\text{if } s = \xi_{lm}.
\end{cases}
\end{equation}
Hence, we infer from \eqref{eq:PsPh}, \eqref{eq:red4} and \eqref{eq:red5} that
\begin{equation}\label{eq:red6}
\begin{aligned}
\|(\pa_s P\Psi_{i\ell}, \pa_s P\Phi_{i\ell})\|_X
&= p \left\| \pa_s(V_i^{p-1}\Phi_{i\ell}) \right\|_{L^{\frac{p+1}{p}}(\Omega)}
+ q \left\| \pa_s(U_i^{q-1}\Psi_{i\ell}) \right\|_{L^{\frac{q+1}{q}}(\Omega)} \\
&= \begin{cases}
O(\delta_{il} \mu^{-1}) &\text{if } s = d_l,\\
O(\delta_{il} \mu^{-2}) &\text{if } s = \xi_{lm}.
\end{cases}
\end{aligned}
\end{equation}
It follows from the condition $(\Psi_\dxi^\ep, \Phi_\dxi^\ep) \in Z_\dxi$, \eqref{eq:nonlin} and \eqref{eq:red6} that
\begin{equation}\label{eq:red3}
\begin{aligned}
\left|\int_{\Omega} \(p V_i^{p-1} \Phi_{i\ell} \pa_s \Phi_\dxi + q U_i^{q-1} \Psi_{i\ell} \pa_s \Psi_\dxi \) \right|
&\le \|(\pa_s P\Psi_{i\ell}, \pa_s P\Phi_{i\ell})\|_X \cdot \left\| (\Psi_\dxi^\ep, \Phi_\dxi^\ep) \right\|_X \\
&= \begin{cases}
o(\delta_{il} \mu^{-1}) &\text{if } s = d_l,\\
o(\delta_{il} \mu^{-2}) &\text{if } s = \xi_{lm}.
\end{cases}
\end{aligned}
\end{equation}

Inserting \eqref{eq:red2} and \eqref{eq:red3} into \eqref{eq:red1}, we conclude that all $c_{i\ell}$'s are zero. Consequently, the claim follows.
In addition, the regularity assertion in Proposition \ref{prop:nonlin} and the bootstrap argument lead to $(U_\dxi^\ep, V_\dxi^\ep) \in (C^2(\ovom))^2$.

\medskip
It only remains to check the positivity assertion for $\beta_1, \beta_2 \ge 0$.
To this end, we modify the nonlinear terms $|v|^{p+1}$ and $|u|^{q+1}$ in the definition \eqref{eq:Iep} of the functional $I_{\ep}$ with $v_+^{p+1}$ and $u_+^{q+1}$, respectively.
If $(\dxi) \in \text{int}(\Lambda)$ is a critical point of $J_{\ep}$ for $\ep \in (0,\ep_0)$,
a similar argument shows that $(U_\dxi^\ep, V_\dxi^\ep)$ is a critical point of the modified functional in $(C^2(\ovom))^2$, which solves
\[\begin{cases}
-\Delta u - \ep (\alpha u + \beta_1 v) = v_+^p \ge 0  &\text{in } \Omega,\\
-\Delta v - \ep (\beta_2 u + \alpha v) = u_+^q \ge 0 &\text{in } \Omega,\\
u, v = 0 &\text{on } \pa \Omega.
\end{cases}\]
If $\beta_1, \beta_2 \ge 0$, the above system is cooperative, so the maximum principle in Theorem 1 of \cite{dFM} implies that the components are positive for $\ep \in (0,\ep_0)$.
\end{proof}

\section{Expansion of the reduced energy} \label{sec:exp}
By the preceding arguments, the existence of critical points of the reduced energy $J_{\ep}$ in $\text{int}(\Lambda)$ implies
that of solutions to \eqref{eq:LEs} with prescribed behavior.
To find them, we compute the asymptotic expansion of $J_{\ep}$ with respect to $\ep$.

Let $A_1,\, \cdots,\, A_5$ be positive numbers defined by
\begin{equation}\label{eq:A1A2}
\begin{cases}
\displaystyle A_1 = \int_{\R^N} U_{1,0}^{q+1},\\
\displaystyle A_2 = \int_{\R^N} U_{1,0}^q,\\
\end{cases}
\quad \text{and} \qquad
\begin{cases}
\displaystyle A_3 = \int_{\R^N} U_{1,0}V_{1,0},\\
\displaystyle A_4 = \int_{\R^N} U_{1,0}^2,\\
\displaystyle A_5 = \int_{\R^N} V_{1,0}^2.
\end{cases}
\end{equation}
Note that $A_1,\, A_2 < \infty$ if $N \ge 3$, $A_3 < \infty$ if $N > 2+\frac{4}{p}$, $A_4 < \infty$ if $N > \frac{4(p+1)}{2p-1}$, $A_5 < \infty$ if $N \ge 5$.
Therefore, if $N \ge 8$ and $p > 1$, then all $A_1, \cdots, A_5$ are finite and ${4(p+1) \over 2p-1} > 2+{4 \over p} > 4$.

Moreover, recalling the numbers $a_{N,p}$, $b_{N,p}$ and $\ga_N$ that appeared in \eqref{eq:dec1} and the definition of $\wth_\dxi$ in \eqref{eq:wth},
let $F_{1\ep},\, F_{2\ep}: \Lambda \to \R$ be functions given as
\begin{equation}\label{eq:F1ep}
F_{1\ep}(\dxi) = \frac{A_2}{p+1} \left[ \({b_{N,p} \over \ga_N}\)^p \sum_{i=1}^k d_i^{N\over q+1} \wth_\dxi(\xi_i)
- a_{N,p} \sum_{i,j=1 \atop i\not=j}^k {d_i^{N \over q+1} d_j^{Np \over q+1} \over |\xi_i-\xi_j|^{(N-2)p-2}} \right]
\end{equation}
and
\begin{multline}\label{eq:F2ep}
F_{2\ep}(\dxi) = \mu^{{N(p-1) \over p+1}} \frac{\beta_1 A_5}{2} \(\sum_{i=1}^k d_i^{{N(p-1) \over p+1}}\) \\
+ \mu^2 \alpha A_3 \(\sum_{i=1}^k d_i^2\)
+ \mu^{{N(q-1) \over q+1}} \frac{\beta_2 A_4}{2} \(\sum_{i=1}^k d_i^{{N(q-1) \over q+1}}\).
\end{multline}
Since ${p-1 \over p+1} < {2 \over N} < {q-1 \over q+1}$, the first term of $F_{2\ep}$ is the largest,
its second term is the second largest, and its last term is the smallest.
We also set $F_{\ep}: \Lambda \to \R$ by
\[F_{\ep}(\dxi) = \mu^{(N-2)p-2} F_{1\ep}(\dxi) - \ep F_{2\ep}(\dxi).\]
Lastly, let $R_{\ep}$ be a quantity such that
\[R_{\ep} = o(\mu^{(N-2)p-2})
+ o\(\ep \left[ \mu^{{N(p-1) \over p+1}} \beta_1 + \mu^2 \alpha + \mu^{{N(q-1) \over q+1}} \beta_2 \right]\)
+ O\(\ep^{q^*} \mu^{Npq^* \over q+1} \beta_2\)\]
uniformly in $\Lambda$.

\begin{prop}\label{prop:exp}
Reduce the value of $\ep_0 > 0$ if needed. Then, for any $\ep \in (0,\ep_0)$,
\begin{equation}\label{eq:exp}
J_{\ep}(\dxi) = {2k \over N} A_1 + F_{\ep}(\dxi) + R_{\ep}
\end{equation}
uniformly in $\Lambda$.
\end{prop}
\noindent In order prove the proposition, we decompose the reduced energy $J_{\ep}$ into three parts
\begin{equation}\label{eq:exp0}
\begin{aligned}
J_{\ep}(\dxi) &= \begin{medsize}
\displaystyle \left[\int_\Omega \nabla \mbP U_\dxi \cdot \nabla \(\sum_{i=1}^k PV_i\)
- \frac1{p+1} \int_\Omega \(\sum_{i=1}^k PV_i\)^{p+1} - \frac1{q+1} \int_\Omega \(\mbP U_\dxi\)^{q+1}\right]
\end{medsize} \\
&\ + \begin{medsize}
\displaystyle \left[- \ep \left\{ \alpha \int_{\Omega} \mbP U_\dxi \(\sum_{i=1}^k PV_i\)
+ \frac{\beta_1}{2} \int_{\Omega} \(\sum_{i=1}^k PV_i\)^2 + \frac{\beta_2}{2} \int_{\Omega} \(\mbP U_\dxi\)^2 \right\} \right]
\end{medsize} \\
&\ + \left[ I_{\ep}\(\mbP U_\dxi + \Psi_\dxi,\, \sum_{i=1}^k PV_i + \Phi_\dxi\) - I_{\ep}\(\mbP U_\dxi,\, \sum_{i=1}^k PV_i\) \right] \\
&=: J_{1\ep}(\dxi) + J_{2\ep}(\dxi) + J_{3\ep}(\dxi)
\end{aligned}
\end{equation}
and estimate each of them, which are the contents of Lemma \ref{lemma:exp02}, \ref{lemma:exp03} and \ref{lemma:exp01}.

\begin{lemma}\label{lemma:exp02}
For any $\ep \in (0,\ep_0)$,
\begin{equation}\label{eq:exp02}
J_{1\ep}(\dxi) = {2k \over N} A_1 + \mu^{(N-2)p-2} \left[F_{1\ep}(\dxi) + o(1)\right]
\end{equation}
uniformly in $\Lambda$.
\end{lemma}
\begin{proof}
Exploiting \eqref{eq:PUPV} and \eqref{eq:mbP}, we get
$$\begin{aligned}
\int_\Omega \(\sum_{i=1}^k PV_i\)^{p+1} &= \int_\Omega \(\sum_{i=1}^k PV_i\)^{p}\(\sum_{i=1}^k PV_i\)  = \int_\Omega \(-\Delta \mbP U_\dxi\) \(\sum_{i=1}^k PV_i\) \\
&= \int_\Omega \mbP U_\dxi \(-\Delta \sum_{i=1}^k PV_i\) = \int_\Omega \mbP U_\dxi \sum_{i=1}^k U_i^q
\end{aligned}$$
and
$$\begin{aligned}
\int_\Omega \nabla \mbP U_\dxi \cdot \nabla \(\sum_{i=1}^k PV_i\)
= \int_\Omega \mbP U_\dxi \(-\Delta \sum_{i=1}^k PV_i\) = \int_\Omega \mbP U_\dxi \sum_{i=1}^k U_i^q.
\end{aligned}$$
Hence,
\begin{equation}\label{eq:exp021}
J_{1\ep}(\dxi) =\frac p{p+1}\int_\Omega  \mbP U_\dxi \(\sum_{i=1}^k U_i^q\) -\frac1{q+1}\int_\Omega  \(\mbP U_\dxi\)^{q+1}.
\end{equation}
We will estimate each of the two integrals on the right-hand side.

Taking into account \eqref{eq:bPU}, we deduce
\begin{align*}
\int_\Omega \mbP U_\dxi \(\sum_{i=1}^k U_i^q\) &=\sum_{i=1}^k\int_\Omega \left[ \sum_{j=1}^k U_j
- \({b_{N,p} \over \ga_N}\)^p \mu^{Np \over q+1} \wth_\dxi + o(\mu^{Np \over q+1}) \right] U_i^q \\
&=\sum_{i=1}^k\int_\Omega U_i^{q+1} + \sum_{i,j=1 \atop i\not=j}^k \int_\Omega U_j U_i^{q}
- \mu^{Np \over q+1} \({b_{N,p} \over \ga_N}\)^p \sum_{i=1}^k \int_\Omega \wth_\dxi U_i^q\\
&\ + o(\mu^{Np \over q+1})\sum_{i=1}^k\int_\Omega U_i^{q}.
\end{align*}
Owing to \eqref{eq:dec1}, \eqref{eq:A1A2} and Lemma \ref{lemma:alg}, we have
\[\int_\Omega U_i^{q+1}(x) dx = \int_{\Omega-\xi_i \over \mu_i} U_{1,0}^{q+1}(y) dy =  A_1 + O\(\int_{\mu^{-1}}^{\infty} \frac{t^{N-1}}{t^{(N-2)p-2)(q+1)}} dt\)
= A_1 + o(\mu^{(N-2)p-2}).\]
Moreover, employing the mean value theorem, \eqref{eq:dec1}, \eqref{eq:dec5}, \eqref{eq:id} and Lemma \ref{lemma:alg}, we obtain
\begin{align*}
\int_\Omega U_j(x) U_i^{q}(x)dx &= \({d_i \over d_j}\)^{N \over q+1} \int_{B^N(0,r\mu_i^{-1})} U_{1,0}^q(y) U_{1,0}\({\xi_i-\xi_j \over \mu_j}\) dy \\
&\ + O\(\int_{B^N(0,r\mu_i^{-1})} U_{1,0}^q(y) \frac{\mu^{(N-2)p-1}}{|\xi_i-\xi_j|^{(N-2)p-1}} |y| dy\) \\
&\ + O\(\mu^{((N-2)p-2)q} \int_{\Omega \setminus B^N(\xi_i,r)} U_j dx\) \\
&= \mu_j^{(N-2)p-2-{N\over q+1}} \mu_i^{N-{Nq\over q+1}} {a_{N,p} \over |\xi_i-\xi_j|^{(N-2)p-2}} A_2 + O(\mu^{(N-2)p-1}) \\
&\ + O(\mu^{(N-2)p-2} \cdot \mu^{((N-2)p-2)q-N}) \\
&= \mu^{(N-2)p-2} a_{N,p} {d_i^{N \over q+1} d_j^{Np \over q+1} \over |\xi_i-\xi_j|^{(N-2)p-2}} A_2 + o(\mu^{(N-2)p-2})
\end{align*}
where $r > 0$ is chosen small. One more application of the mean value theorem, with the help of Lemma \ref{lemma:wtg} in this time, shows
\begin{align*}
\mu^{Np \over q+1} \int_\Omega \wth_\dxi(x) U_i^q(x) dx &= \mu^{Np \over q+1} \mu_i^{N-{Nq\over q+1}} \int_{\Omega-\xi_i \over \mu_i} \wth_\dxi(\xi_i + \mu_i y) U_{1,0}^q(y) dy \\
&= \mu^{Np \over q+1} \mu_i^{N\over q+1} \wth_\dxi(\xi_i) \int_{\Omega-\xi_i \over \mu_i} U_{1,0}^q(y) dy + O(\mu^{(N-2)p-1}) \\
&= \mu^{(N-2)p-2} d_i^{N \over q+1} \wth_\dxi(\xi_i) A_2 + o(\mu^{(N-2)p-2}).
\end{align*}
Therefore,
\begin{multline}\label{eq:exp022}
\int_\Omega \mbP U_\dxi \(\sum_{i=1}^k U_i^q\) = kA_1 + \mu^{(N-2)p-2} A_2 \\
\times \left[a_{N,p} \sum_{i,j=1 \atop i\not=j}^k
{d_i^{N \over q+1} d_j^{Np\over q+1} \over |\xi_i-\xi_j|^{(N-2)p-2}}
- \({b_{N,p} \over \ga_N}\)^p \sum_{i=1}^k d_i^{N \over q+1} \wth_\dxi(\xi_i) \right] + o(\mu^{(N-2)p-2}).
\end{multline}

On the other hand, if we write
$$\int_\Omega \(\mbP U_\dxi\)^{q+1} = \sum_{i=1}^k \int_{B^N(\xi_i,\rho)} \(\mbP U_\dxi\)^{q+1} + \int_{\Omega \setminus \cup_{i=1}^k B^N(\xi_i,\rho)} \(\mbP U_\dxi\)^{q+1}$$
for some small $\rho > 0$, then
\[0 \le \int_{\Omega \setminus \cup_{i=1}^k B^N(\xi_i,\rho)}\(\mbP U_\dxi\)^{q+1}
\le C \sum_{i=1}^k \int_{\Omega \setminus B^N(\xi_i,\rho)} U_i^{q+1} + O(\mu^{Np})
= o(\mu^{(N-2)p-2}).\]
Furthermore,
\begin{align}
&\ \int_{B^N(\xi_i,\rho)} \(\mbP U_\dxi\)^{q+1} \nonumber \\
&= \int_{B^N(\xi_i,\rho)} \left[ U_i+\sum_{j=1 \atop j\not=i}^k U_j - \mu^{Np \over q+1}
\({b_{N,p} \over \ga_N}\)^p \wth_\dxi + o(\mu^{Np \over q+1}) \right]^{q+1} \nonumber \\
&= \int_{B^N(\xi_i,\rho)} U_i^{q+1}dx + (q+1) \int_{B^N(\xi_i,\rho)}
\left[ \sum_{j=1\atop j\not=i}^k U_j
- \mu^{Np \over q+1} \({b_{N,p} \over \ga_N}\)^p \wth_\dxi + o(\mu^{Np \over q+1}) \right] U_i^{q} \nonumber \\
&\ + O(\mu^{Np}) + O(\mu^{2((N-2)p-2)}) \label{eq:exp023} \\
&= A_1 + \mu^{(N-2)p-2} (q+1) A_2 \left[ a_{N,p} \sum_{j=1 \atop j\not=i}^k {d_i^{N\over q+1} d_j^{Np \over q+1} \over |\xi_i-\xi_j|^{(N-2)p-2}}
- \({b_{N,p} \over \ga_N}\)^p \sum_{i=1}^k  d_i^{N\over q+1} \wth_\dxi(\xi_i) \right] \nonumber \\
&\ + o(\mu^{(N-2)p-2}). \nonumber
\end{align}

Collecting \eqref{eq:exp021}-\eqref{eq:exp023} and using \eqref{eq:hyper}, we establish \eqref{eq:exp02}.
\end{proof}

\begin{lemma}\label{lemma:exp03}
For any $\ep \in (0,\ep_0)$,
\[J_{2\ep}(\dxi) = - \ep \left[ F_{2\ep}(\dxi) + o\(\mu^{{N(p-1) \over p+1}} \beta_1 + \mu^2 \alpha + \mu^{{N(q-1) \over q+1}} \beta_2\) \right]\]
uniformly in $\Lambda$.
\end{lemma}
\begin{proof}
A direct calculation gives
\begin{equation}\label{eq:exp031}
\begin{aligned}
\alpha \int_{\Omega} \mbP U_\dxi \(\sum_{i=1}^k PV_i\) &= \alpha \left[\sum_{i=1}^k \int_{\Omega} U_iV_i + O(\mu^{(N-2)p-2}) \right] \\
&= \mu^2 \alpha \left[A_3 \(\sum_{i=1}^k d_i^2\) + o(1)\right],
\end{aligned}
\end{equation}
\begin{equation}\label{eq:exp033}
\begin{aligned}
\frac{\beta_1}{2} \int_{\Omega} \(\sum_{i=1}^k PV_i\)^2
&= \frac{\beta_1}{2} \left[\sum_{i=1}^k \int_{\Omega} V_i^2 + O(\mu^{2N \over q+1})\right] \\
&= \mu^{{N(p-1) \over p+1}} \beta_1 \left[\frac{A_5}{2} \(\sum_{i=1}^k d_i^{{N(p-1) \over p+1}}\) + o(1)\right],
\end{aligned}
\end{equation}
and
\begin{equation}\label{eq:exp032}
\begin{aligned}
\frac{\beta_2}{2} \int_{\Omega} \(\mbP U_\dxi\)^2
&= \frac{\beta_2}{2} \left[\sum_{i=1}^k \int_{\Omega} U_i^2 + O(\mu^{2Np \over q+1})\right] \\
&= \mu^{{N(q-1) \over q+1}} \beta_2 \left[\frac{A_4}{2} \(\sum_{i=1}^k d_i^{{N(q-1) \over q+1}}\) + o(1)\right].
\end{aligned}
\end{equation}
Combining \eqref{eq:exp031}-\eqref{eq:exp032}, we conclude the proof.
\end{proof}

\begin{lemma}\label{lemma:exp01}
For any $\ep \in (0,\ep_0)$,
\begin{equation}\label{eq:exp01}
J_{3\ep}(\dxi) = o(\mu^{(N-2)p-2}) + O\(\ep^{q^*} \left[ \mu^{N(p-1) \over p+1} \beta_1 + \mu^2 \alpha + \(\mu^{Np \over q+1} + \mu^{N(q-1) \over q+1}\) \beta_2 \right]^{q^*}\)
\end{equation}
uniformly in $\Lambda$.
\end{lemma}
\begin{proof}
By Proposition \ref{prop:nonlin}, we have
\begin{multline*}
I_{\ep}'\(\mbP U_\dxi + \Psi_\dxi,\, \sum_{i=1}^k PV_i + \Phi_\dxi\)(\Psi_\dxi, \Phi_\dxi) \\
= \sum_{i=1}^k \sum_{\ell=0}^N c_{i\ell} \int_{\R^N} \(p V_i^{p-1} \Phi_{i\ell} \Phi + q U_i^{q-1} \Psi_{i\ell} \Psi\) dx = 0.
\end{multline*}
Hence Taylor's theorem implies
\begin{equation}\label{eq:exp011}
J_{3\ep}(\dxi) = - \int_0^1 \theta I_{\ep}''\(\mbP U_\dxi + \theta\Psi_\dxi,\, \sum_{i=1}^k PV_i + \theta\Phi_\dxi\)(\Psi_\dxi, \Phi_\dxi)^2 d\theta.
\end{equation}
In view of \eqref{eq:nonlin} and \eqref{eq:id},
\begin{align*}
&\ \sup_{\theta \in [0,1]} \left| I_{\ep}''\(\mbP U_\dxi + \theta\Psi_\dxi,\, \sum_{i=1}^k PV_i + \theta\Phi_\dxi\)(\Psi_\dxi, \Phi_\dxi)^2 \right| \\
&\le C \left[ \int_{\Omega} |\nabla \Psi_\dxi|^{p^*} + \int_{\Omega} |\nabla \Phi_\dxi|^{q^*}
+ \int_{\Omega} |\Psi_\dxi|^{q+1} + \int_{\Omega} |\Phi_\dxi|^{p+1} \right. \\
&\qquad \left. + \int_{\Omega} \mbP U_\dxi^{q-1} \Psi_\dxi^2 + \int_{\Omega} \(\sum_{i=1}^k PV_i\)^{p-1} \Phi_\dxi^2
+ \ep \(\int_{\Omega} \Psi_\dxi^2 + \int_{\Omega} \Phi_\dxi^2\) \right] \\
&\le C\left\| (\Psi_\dxi^\ep, \Phi_\dxi^\ep) \right\|_X^{q^*} \quad \text{(for $q \ge p \ge 1$ \text{ and } $1 < q^* < 2 < p^*$)} \\
&\le C\|\mce_\dxi\|_X^{q^*} \\
&= o(\mu^{(N-2)p-2}) + O\(\ep^{q^*} \left[ \mu^{N(p-1) \over p+1} \beta_1 + \mu^2 \alpha + \(\mu^{Np \over q+1} + \mu^{N(q-1) \over q+1}\) \beta_2 \right]^{q^*}\).
\end{align*}
See Lemma \ref{lemma:alg} for the equality on the last line.
Plugging this estimate into \eqref{eq:exp011}, we discover \eqref{eq:exp01}.
\end{proof}

\begin{proof}[Completion of the proof of Proposition \ref{prop:red}]
Estimate \eqref{eq:exp} is a direct consequence of \eqref{eq:exp0}, and Lemmas \ref{lemma:exp02}, \ref{lemma:exp03} and \ref{lemma:exp01}.
\end{proof}

\section{Completion of the proof of the main theorems} \label{sec:comp}
\subsection{Proof of Theorem \ref{thm:main1}}\label{subsec:proof11}
By Propositions \ref{prop:red} and \ref{prop:exp}, we are led to find an interior critical point of the reduced energy $J_\ep$ on the configuration space $\Lambda$ which can be rewritten as
\begin{equation}\label{eq:fdef}
\begin{aligned}
J_\ep(\dxi) &= C_0 + C_1 \mu^{(N-2)p-2}d^{(N-2)p-2} \tta(\xi)\\
&\ - \ep \(C_2 \beta_1 \mu^{{N(p-1) \over p+1}} d^{{N(p-1) \over p+1}}
+ C_3 \alpha \mu^2 d^2 + C_4 \beta_2 \mu^{{N(q-1) \over q+1}} d^{{N(q-1) \over q+1}}\) + \text{(h.o.t.)}
\end{aligned}
\end{equation}
where $C_1, \cdots, C_4$ are positive constants and (h.o.t.) stands for a higher order term.
\begin{itemize}
\item[{\it (i)}] \textsc{Case $\beta_1>0$}: The lowest order term in the bracket on the right-hand side of \eqref{eq:fdef} is the $\beta_1$-term, so we take the rate of $\mu$ to be
$$\mu^{(N-2)p-2} = \ep\mu^{{N(p-1) \over p+1}}\ \Rightarrow\ \mu = \ep^{p+1\over (N-2)p^2-4p+N-2}$$
where $(N-2)p^2-4p+(N-2) > 0$ for all $N \ge 4$ and $p > 1$.
Then \eqref{eq:fdef} reduces to
$$J_\ep(d,\xi) = C_0 + \mu^{(N-2)p-2} \underbrace{\(C_1d^{(N-2)p-2} \tta(\xi) - C_2\beta_1 d^{{N(p-1) \over p+1}}\)}_{=: \mfj_1(d,\xi)} + \text{(h.o.t.)}.$$
By picking sufficiently small $\delta_1, \delta_2 > 0$ in the definition \eqref{eq:Lambda} of $\Lambda$,
we can make the function $\mfj_1$ have a strict minimum point in $\text{int}(\Lambda)$.
Thus $J_\ep$ has a minimum point in $\text{int}(\Lambda)$ provided that $\ep > 0$ is small enough.

\item[{\it (ii)}] \textsc{Case $\beta_1=0$ and $\alpha>0$}:
The lowest order term in the bracket on the right-hand side of \eqref{eq:fdef} is the $\alpha$-term, so we take the rate of $\mu$ to be
$$\mu^{(N-2)p-2} = \ep\mu^2\ \Rightarrow\ \mu = \ep^{1 \over (N-2)p-4}$$
where $(N-2)p > 4$ for all $N \ge 6$ and $p > 1$. Then \eqref{eq:fdef} reduces to
$$J_\ep(d,\xi) = C_0 + \mu^{(N-2)p-2} \underbrace{\(C_1d^{N(p+1) \over q+1} \tta(\xi) - C_3 \alpha d^2\)}_{:=\mfj_2(d,\xi)} + \text{(h.o.t.)}.$$
The function $\mfj_2$ has a a strict minimum point in $\text{int}(\Lambda)$,
and so $J_\ep$ has a minimum point provided that $\ep > 0$ is small enough.

\item[{\it (iii)}] \textsc{Case $\beta_1=\alpha=0$ and $\beta_2>0$}: Reminding \eqref{eq:id}, we take the rate of $\mu$ to be
$$\mu^{(N-2)p-2} = \ep\mu^{N(q-1)\over q+1}\ \Rightarrow\ \mu = \ep^{q+1\over N(p-q+2)},$$
where $p-q+2>0$ for $N \ge 8$ and $p > 1$ by Lemma \ref{lemma:alg2}.
Then \eqref{eq:fdef} reduces to
$$J_\ep(d,\xi) = C_0 + \mu^{(N-2)p-2} \underbrace{\(C_1d^{N(p+1)\over q+1} \tta(\xi) - C_4 \beta_2d^{{N(q-1) \over q+1}}\)}_{:=\mfj_3(d,\xi)}+\text{(h.o.t.)}.$$
The function $\mfj_3$ has a a strict minimum point in $\text{int}(\Lambda)$,
and so $J_\ep$ has a minimum point provided that $\ep > 0$ is small enough.
\end{itemize}

\subsection{Proof of Theorem \ref{thm:main2}}\label{subsec:proof12}
Let $\wth^\eta_\dxi$ be the function introduced in \eqref{eq:wth} for the dumbbell-shaped domain $\Omega_{\eta}$ and
\begin{equation}\label{eq:fmu}
F^\eta(\dxi) = \({b_{N,p} \over \ga_N}\)^p\sum_{i=1}^kd_i^{N \over q+1}\wth^\eta_\dxi(\xi_i)
- a_{N,p}\sum_{i,j=1 \atop i\not=j}^k{d_i^{N \over q+1}d_j^{Np \over q+1}\over |\xi_i-\xi_j|^{(N-2)p-2}}
\end{equation}
(compare with \eqref{eq:F1ep}). We agree that $\wth_\dxi^0$, $F^0$, and also $\wtg_\dxi^0$ are related to the disconnected domain $\Omega_0 = \cup_{i=1}^l \Omega^*_i.$
\begin{lemma}
Let $\Lambda_0$ be the configuration space $\Lambda$ defined in \eqref{eq:Lambda} related to $\Omega_0$. Then
\begin{equation}\label{cru}
\lim_{\eta \to 0} F^\eta(\dxi) = F^0(\dxi) \quad \text{uniformly on } \Lambda_0.
\end{equation}
\end{lemma}
\begin{proof}
It is obvious that \eqref{cru} follows by
\begin{equation}\label{cru2}
\lim_{\eta\to0} \wth^\eta_\dxi(\xi) = \wth_\dxi^0(\xi) \quad  \text{uniformly on } \Lambda_0.
\end{equation}
Let us prove \eqref{cru2}.

First of all, we recall from Lemma 3.2 in \cite{MP} that
\begin{equation}\label{91}
H_{\Omega_{\eta}}(x,y) \nearrow H_{\Omega_0}(x,y) \quad C^1\hbox{-uniformly on compact sets of $\Omega_0 \times \Omega_0$ as $\eta\to0$}.
\end{equation}

The function $\wth^\eta_\dxi$ solves the problem
\begin{equation}\label{9}
\begin{cases}
\begin{aligned}
\displaystyle -\Delta \wth_\dxi^\eta &= \sum_{i=1}^k \(\frac{d_i^{N \over q+1} \ga_N}{|\cdot-\xi_i|^{N-2}}\)^p
- \left[\sum_{i=1}^k d_i^{N \over q+1} \({\ga_N \over |\cdot-\xi_i|^{N-2}}-H_{\Omega_{\eta}}(\cdot,\xi_i)\) \right]^p \\
&=: f^\eta
\end{aligned}
&\text{in } \Omega_{\eta}, \\
\displaystyle \wth_\dxi^\eta = \sum_{i=1}^k d_i^{Np \over q+1} \frac{\tga_{N,p}}{|\cdot-\xi_i|^{(N-2)p-2}} &\text{on } \pa \Omega_{\eta}
\end{cases}
\end{equation}
(compare with \eqref{eq:wth2}).

By a comparison argument, we observe that
\begin{equation}\label{93}
\begin{cases}
\wth^\eta_\dxi(x) \text{ is monotone increasing as } \eta \to 0\\
\wth^\eta_\dxi(x)\le \wth_\dxi^0(x)
\end{cases}
\quad \hbox{for any}\ x \in \Omega_0.
\end{equation}
Moreover, given a fixed small number $\eta_0 > 0$, standard regularity estimates and the Sobolev embedding theorem ensure that for any compact subset $\Omega'$ of $\Omega_0$,
there exists a constant $C > 0$ depending only on $N$, $p$, $\Lambda_0$, $\eta_0$, $\Omega_0$ and $\Omega'$ such that
\begin{equation}\label{92}
\left\|\wth^\eta_\dxi\right\|_{C^{1,\sigma}(\Omega')} \le C \left\|\wth^\eta_\dxi\right\|_{W^{2,s}(\Omega')}
\le C \(\left\|\wth^\eta_\dxi\right\|_{L^{\infty}(\Omega_0)} + \|f^\eta\|_{L^s(\Omega_0)}\) \le C
\end{equation}
for all $\eta \in (0,\eta_0)$, some $s > N$, and $\sigma = 1 - \frac{N}{s}$. Indeed, by \eqref{93}, we have
\[\left\|\wth^\eta_\dxi\right\|_{L^\infty(\Omega_0)} \le C \quad \text{and} \quad \|f^\eta\|_{L^s(\Omega_0)} \le C\]
for $s \in \big(N, {N \over (N-2)(p-1)}\big)$ which is an nonempty interval for $p < {N-1\over N-2}$.
By virtue of \eqref{93} and \eqref{92}, there exists a function $\mfh$ in $\Omega_0$ such that
\[\wth^\eta_\dxi(x) \nearrow \mfh(x) \quad \text{$C^1$-uniformly in compact sets of $\Omega_0$ as $\eta \to 0.$}\]
Now, we see from \eqref{9} that
\[\begin{cases}
\displaystyle -\Delta \mfh = f^0 := \lim_{\eta \to 0} f^\eta &\text{in } \Omega_0, \\
\displaystyle \mfh = \sum_{i=1}^k d_i^{Np \over q+1} \frac{\tga_{N,p}}{|\cdot-\xi_i|^{(N-2)p-2}} &\text{on } \pa \Omega_0 \setminus \left\{(x',x_N) \in \R \times \R^N: |x'| = 0 \right\}.
\end{cases}\]
By \eqref{91}, we immediately deduce that $\mfh = \wth_\dxi^0$ and \eqref{cru2} holds.
 \end{proof}
\begin{lemma}
It holds true that
\begin{equation}\label{s1}
F^0(\dxi) = \({b_{N,p} \over \ga_N}\)^p \sum_{i=1}^kd_i^{N(p+1) \over q+1} \tta_{\Omega^*_i}(\xi_i)
\end{equation}
for any $(\dxi) \in (0,\infty)^k \times (\Omega^*_1 \times \dots \times \Omega^*_k)$.
\end{lemma}
\begin{proof}
First of all, we remark that given points $\xi_i\in\Omega^*_i$,
the Dirichlet Green's function $G_{\Omega_0}$ of $-\Delta$ in the disconnected domain $\Omega_0$ satisfies
\begin{equation}\label{3}
G_{\Omega_0}(x,\xi_i) = 0\ \hbox{if}\ x \not\in \Omega^*_i
\quad \hbox{and} \quad
G_{\Omega_0}(x,\xi_i) = G_{\Omega^*_i}(x,\xi_i) \quad \hbox{if}\  x \in \Omega^*_i.
\end{equation}
By \eqref{eq:wtg} and \eqref{3},
\begin{equation}\label{5}
\wtg_\dxi^0(x) = d_i^{Np\over q+1} \wtg_{\Omega^*_i}(x,\xi_i) \quad \hbox{if}\ x \in \Omega^*_i.
\end{equation}
Hence, it follows from \eqref{eq:wth0}, \eqref{eq:wth} and \eqref{5} that
\[\wth_\dxi^0(x)= d_i^{Np \over q+1} \wth_{\Omega^*_i}(x,\xi_i) + \sum_{j \ne i} d_j^{Np \over q+1} {\tga_{N,p} \over |x-\xi_j|^{(N-2)p-2}} \quad \hbox{if}\ x \in \Omega^*_i,\]
and in particular,
\begin{equation}\label{61}
\wth_\dxi^0(\xi_i)= d_i^{Np \over q+1} \tta_{\Omega^*_i}(\xi_i) + \sum_{j \ne i} d_j^{Np \over q+1} {\tga_{N,p} \over |\xi_i-\xi_j|^{(N-2)p-2}}.
\end{equation}
Putting \eqref{61} into \eqref{eq:fmu} with $\eta = 0$, and applying \eqref{eq:tga} and \eqref{eq:dec2}, we conclude that
\begin{align*}
F^0(\dxi) 
&= \({b_{N,p} \over \ga_N}\)^p \(\sum_{i=1}^k d_i^{N(p+1) \over q+1} \tta_{\Omega^*_i}(\xi_i) + \sum_{i,j=1 \atop i\not=j}^k d_i^{N \over q+1}d_j^{Np \over q+1} {\tga_{N,p} \over |\xi_i-\xi_j|^{(N-2)p-2}}\) \\
&\ - a_{N,p} \sum_{i,j=1 \atop i\not=j}^k {d_i^{N \over q+1} d_j^{Np \over q+1} \over |\xi_i-\xi_j|^{(N-2)p-2}} \\
&= \({b_{N,p} \over \ga_N}\)^p \sum_{i=1}^k d_i^{N(p+1) \over q+1} \tta_{\Omega^*_i}(\xi_i),
\end{align*}
and \eqref{s1} follows.
\end{proof}

\begin{proof}[Completion of the proof of Theorem \ref{thm:main2}]
Here, we only treat the case that $\beta_1>0$. The other case can be done similarly.

By Propositions \ref{prop:red} and \ref{prop:exp}, we only need to find an interior critical point of the reduced energy $J_\ep$ on the corresponding configuration space $\Lambda_{\eta}$ which can be rewritten as
\begin{equation}\label{eq:Jep2}
J_\ep(\dxi) = C_0 + \mu^{(N-2)p-2} \underbrace{\(C_1 F^\eta(\dxi) - C_2\beta_1 \sum_{i=1}^k d_i^{{N(p-1) \over p+1}}\)}_{=: \mfj^\eta(\dxi)} + \text{(h.o.t.)}
\end{equation}
where $F^\eta$ is defined in \eqref{eq:fmu} and $C_0$, $C_1$ and $C_2$ are positive constants.
By \eqref{cru} and \eqref{s1},
$$\mfj^\eta(\dxi) \to \mfj^0(\dxi) := \sum_{i=1}^k \(C_1 d_i^{N(p+1) \over q+1} \tta_{\Omega^*_i}(\xi_i) - C_2\beta_1 d_i^{N(p-1) \over p+1}\)$$
uniformly on $\Lambda_0$ as $\ep \to 0$.
The function $\mfj^0$ has a strict minimum point $(\mbd^0,\mbxi^0) \in (0,\infty)^k \times (\Omega^*_1 \times \dots \times \Omega^*_k)$.
Thus $\mfj^\eta$ also has a strict minimum point $(\mbd^{\eta},\mbxi^{\eta}) \in (0,\infty)^k \times (\Omega^*_1 \times \dots \times \Omega^*_k)$ provided that $\eta$ is small enough,
and finally, the energy $J_\ep$ has a minimum point provided that $\ep$ is small enough.
\end{proof}

\section{Solutions with sign-changing blow-up points} \label{sec:nod}
Our strategy to build solutions to \eqref{eq:LEs} with sign-changing blow-up points is the same as what we used in the previous sections. In particular, we will look for solutions
\[(U_\dxi^\ep, V_\dxi^\ep) := \(\mbP U_\dxi + \Psi_\dxi,\, \sum_{i=1}^k \lambda_i PV_i + \Phi_\dxi\) \in X\]
where $(\dxi)$ belongs to the configuration space $\Lambda$ defined in \eqref{eq:Lambda},
$\lambda_i=+1$ or $\lambda_i=-1$ if the blow-up point $\xi_i$ is positive or negative, respectively, the function $\mbP U_\dxi$ solves
\[\begin{cases}
\displaystyle -\Delta \mbP U_\dxi = \left|\sum_{i=1}^k \lambda_i PV_i\right|^{p-1}\(\sum_{i=1}^k \lambda_i PV_i\) &\text{in } \Omega,\\
\mbP U_\dxi = 0 &\text{on } \pa \Omega,
\end{cases}\]
and $(\Psi_\dxi,\Phi_\dxi)$ is a higher order term. It is important to point out that in this case $\mbP U_\dxi $ has the expansion
$$\mbP U_\dxi(x) = \sum_{i=1}^k\lambda_i U_i - \mu^{Np \over q+1} \({b_{N,p} \over \ga_N}\)^p \wth_\dxi(x) + o(\mu^{Np \over q+1})$$
(see Lemma \ref{lemma:wtg}), where $\wtg_\dxi$ solves
$$\begin{cases}
\displaystyle -\Delta \wtg_\dxi=\left|\sum_{i=1}^k \lambda_i d_i^{N\over q+1}G_{\Omega }(\cdot,\xi_i)\right|^{p-1}\(\sum_{i=1}^k \lambda_i d_i^{N\over q+1}G_{\Omega }(\cdot,\xi_i)\) &\hbox{in}\ \Omega, \\
\wtg_\dxi=0 &\hbox{on}\ \pa\Omega,
\end{cases}$$
and $\wth_\dxi$ is defined as
$$\wth_\dxi(x) = \sum_{i=1}^k\lambda_id_i^{Np\over q+1} {\tga_{N,p}\over |x-\xi_i|^{(N-2)p-2}}-\wtg_\dxi(x).$$

Arguing as in the case of solutions with positive blow-up points,
we can prove results analogous to Proposition \ref{prop:exp} and Proposition \ref{prop:red},
which enable us to obtain a solution to \eqref{eq:LEs} via
finding a critical point of the reduced energy which in this case can be rewritten as
\begin{equation}\label{nodal}
J_\ep(\dxi) = C_0 + \mu^{(N-2)p-2} \wtf_1(\dxi) - \ep \wtf_2(\dxi) + \text{(h.o.t.)},
\end{equation}
where
\[\wtf_1 (\dxi): = \frac{A_2}{p+1} \left[ \({b_{N,p} \over \ga_N}\)^p \sum_{i=1}^k \lambda_i d_i^{N \over q+1} \wth_\dxi(\xi_i)
- a_{N,p} \sum_{i,j=1\atop i\not=j}^k \lambda_i\lambda_j {d_i^{N \over q+1}d_j^{Np \over q+1} \over |\xi_i-\xi_j|^{(N-2)p-2}} \right]\]
and
$$\wtf_2(\dxi) = \mu^{{N(p-1) \over p+1}} \frac{\beta_1 A_5}{2} \(\sum_{i=1}^k d_i^{{N(p-1) \over p+1}}\)
+ \mu^2 \alpha A_3 \(\sum_{i=1}^k d_i^2\)
+ \mu^{{N(q-1) \over q+1}} \frac{\beta_2 A_4}{2} \(\sum_{i=1}^k d_i^{{N(q-1) \over q+1}}\).$$
(compare with \eqref{eq:F1ep}, \eqref{eq:F2ep}).

Let us consider the simplest case of one positive blow-up point and one negative blow-up point, i.e. $\lambda_1=-\lambda_2=+1$.
Then, $\wtf_1$ reduces to
$$\wtf_1(\dxi) = \frac{A_2}{p+1} \({b_{N,p} \over \ga_N}\)^p \( d_1^{N \over q+1}\wth_\dxi(\xi_1)+d_2^{N \over q+1}\wth_\dxi(\xi_2)
+ 2a_{N,p} {d_1^{N \over q+1}d_2^{Np \over q+1} \over |\xi_1-\xi_2|^{(N-2)p-2}}\).$$
If we are able to prove that
\begin{equation}\label{conj}
\text{{\it there exists} } M>0 \text{ {\it such that} } \wtf_1(\dxi)\ge M\(d_1^{N \over q+1}+d_2^{N \over q+1}\) \text{ {\it for any} } \mbxi\in\Omega\times\Omega,
\end{equation}
then, arguing as in Section \ref{subsec:proof11}, we can find a minimum point of the reduced energy $J_\ep$ and so a solution with desired blow-up points to our problem.
In the case of the single equation \eqref{bn}, the condition corresponding to \eqref{conj} is satisfied,
so there always exists a sign-changing solution with one positive blow-up point and one negative blow-up point.
However, in the the case of the system, the proof of \eqref{conj} requires some additional works because of the way that $\wtf_1$ is defined.
We plan to return to this issue in a future work.

\section{Slightly subcritical problems} \label{sec:subc}
Modifying our method, we can also find solutions to slightly subcritical systems
\begin{equation}\label{eq:LEsub}
\begin{cases}
-\Delta u = v^{p-\alpha\ep} &\text{in } \Omega,\\
-\Delta v = u^{q-\beta\ep} &\text{in } \Omega,\\
u, v > 0 &\text{in } \Omega,\\
u, v = 0 &\text{on } \pa \Omega
\end{cases}
\end{equation}
where $\Omega$ is a smooth bounded domain in $\R^N,$ $N \ge 3$, $\ep > 0$ is a small parameter,
$(p,q)$ is a pair of positive numbers on the critical hyperbola \eqref{eq:hyper}, and $(\alpha, \beta)$ is a pair of numbers such that
\begin{equation}\label{eq:hyper2}
\frac{\alpha}{(p+1)^2} + \frac{\beta}{(q+1)^2} > 0.
\end{equation}

\medskip
The followings are the main theorems in this section.
\begin{thm}\label{thm:sub1}
Assume that $N \ge 4$, $p \in (1, \frac{N-1}{N-2})$, $(p,q)$ satisfies \eqref{eq:hyper}, and $(\alpha, \beta)$ is a pair of positive numbers.
Then there exists a small number $\ep_0 > 0$ depending only on $N$, $p$, $\Omega$, $\alpha$ and $\beta$
such that for any $\ep\in(0,\ep_0)$, system \eqref{eq:LEsub} has a solution in $(C^2(\ovom))^2$ which blows-up at one point in $\Omega$ as $\ep \to 0$.
\end{thm}
\begin{thm}\label{thm:sub2}
Assume that $N \ge 4$, $p \in (1, \frac{N-1}{N-2})$, $(p,q)$ satisfies \eqref{eq:hyper}, $(\alpha, \beta)$ is a pair of positive numbers, and $k \in \{1, \cdots, l\}$.
Then there exist small numbers $\eta_0,\, \ep_0 > 0$ such that for any $\ep \in (0,\ep_0)$ and $\eta \in (0,\eta_0)$,
system \eqref{eq:LEsub} with $\Omega = \Omega_{\eta}$ has ${l \choose k}$ solutions in $(C^2(\ovom))^2$ which blow-up at $k$ points as $\ep \to 0$.
\end{thm}
\begin{rmk}
(1) In \cite{Ge, CK}, the authors studied asymptotic behavior of ground state solutions to \eqref{eq:LEsub} when $\alpha = 0$, namely, when only the exponent of $u$ moves.
Our result permits the situation that each exponent of $u$ and $v$ moves simultaneously as $\ep \to 0$.

\medskip \noindent (2) We expect that the above theorems are true for any pair $(\alpha,\beta)$ satisfying \eqref{eq:hyper2}
regardless of the sign of $\alpha$ and $\beta$ (provided that the other assumptions are kept).
For such $(\alpha,\beta)$, the pair $(p+1-\alpha\ep, q+1-\beta\ep)$ is {\it slightly subcritical} for sufficiently small $\ep > 0$, that is,
\[\frac{1}{p+1-\alpha\ep} + \frac{1}{q+1-\beta\ep} > \frac{N-2}{N}\]
and its left-hand side tends to $\frac{N-2}{N}$ as $\ep \to 0$.
To verify our claim, we have to replace the function space $X_{p,q}$ in \eqref{eq:Xpq} with its variant
\[X_{p,q,\alpha,\beta,\ep} := \left\{(u,v) \in X_{p,q} : u \in L^{q+1-\beta\ep}(\Omega),\, v \in L^{p+1-\beta\ep}(\Omega) \right\}\]
endowed with the norm
\[\|(u,v)\|_{X_{p,q,\alpha,\beta,\ep}} := \|\Delta u\|_{L^{\frac{p+1}{p}}(\Omega)} + \|\Delta v\|_{L^{\frac{q+1}{q}}(\Omega)} + \|u\|_{L^{q+1-\beta\ep}(\Omega)} + \|v\|_{L^{p+1-\beta\ep}(\Omega)}.\]
In the above theorems, we have opted to impose $\alpha, \beta > 0$ in order not to make further technical computations.
In view of \eqref{eq:Sob}, the space $X_{p,q,\alpha,\beta,\ep}$ is equivalent to $X_{p,q}$ for $\alpha, \beta > 0$.
\end{rmk}

To prove Theorems \ref{thm:sub1} and \ref{thm:sub2}, we keep using the approximate solutions $\(\mbP U_\dxi,\, \sum_{i=1}^k PV_i\)$ constructed in Subsection \ref{subsec:app}.
Then, for any $N \ge 4$ and $p \in (1, \frac{N-1}{N-2})$, all the results in Sections \ref{sec:set} and \ref{sec:lin} continue to hold after suitable modifications.
Especially, we do not require the dimensional restriction $N \ge 8$ as in Theorems \ref{thm:main1} and \ref{thm:main2},
because it is necessary only when we treat the linear terms in \eqref{eq:LEs}.

Note that the energy functional $I_{\ep}$ needs to be redefined as
\[I_{\ep}(u,v) = \int_{\Omega} \nabla u \cdot \nabla v - \frac{1}{p+1-\alpha\ep} \int_{\Omega} v_+^{p+1-\alpha\ep}
- \frac{1}{q+1-\beta\ep} \int_{\Omega} u_+^{q+1-\beta\ep}
\quad \text{for } (u,v) \in X.\]
Then its reduced energy $J_{\ep}$ defined via \eqref{eq:Jep} is decomposed into
\begin{align*}
J_{\ep}(\dxi) &= \begin{medsize}
\displaystyle \left[\int_\Omega \nabla \mbP U_\dxi \cdot \nabla \(\sum_{i=1}^k PV_i\)
- \frac1{p+1} \int_\Omega \(\sum_{i=1}^k PV_i\)^{p+1} - \frac1{q+1} \int_\Omega \(\mbP U_\dxi\)^{q+1}\right]
\end{medsize} \\
&\  +  \left[\begin{medsize}
\displaystyle \left\{ \frac1{p+1} \int_\Omega \(\sum_{i=1}^k PV_i\)^{p+1} - \frac{1}{p+1-\alpha\ep} \int_\Omega \(\sum_{i=1}^k PV_i\)^{p+1-\alpha\ep} \right\}
\end{medsize} \right. \\
&\qquad \left. + \begin{medsize}
\displaystyle \left\{ \frac1{q+1} \int_\Omega \(\mbP U_\dxi\)^{q+1} - \frac1{q+1-\beta\ep} \int_\Omega \(\mbP U_\dxi\)^{q+1-\beta\ep} \right\}
\end{medsize} \right] \\
&\ + \left[ I_{\ep}\(\mbP U_\dxi + \Psi_\dxi,\, \sum_{i=1}^k PV_i + \Phi_\dxi\) - I_{\ep}\(\mbP U_\dxi,\, \sum_{i=1}^k PV_i\) \right] \\
&=: J_{4\ep}(\dxi) + J_{5\ep}(\dxi) + J_{6\ep}(\dxi)
\end{align*}
(compare with \eqref{eq:exp0}).

Let $A_1$ and $F_{1\ep}: \Lambda \to \R$ be the positive number and the function defined in \eqref{eq:A1A2} and \eqref{eq:F1ep}, respectively.
Set numbers $A_6$ and $A_7$ by
\[A_6 = \int_{\R^N} U_{1,0}^{q+1} \log U_{1,0}
\quad \text{and} \quad
A_7 = \int_{\R^N} V_{1,0}^{p+1} \log V_{1,0},\]
which are finite if $N \ge 3$, and a function $F_{3\ep}: \Lambda \to \R$ by
\[F_{3\ep}(\dxi) = - \left[\frac{\alpha}{(p+1)^2} + \frac{\beta}{(q+1)^2}\right] NA_1 \log(d_1 \cdots d_k).\]
\begin{lemma}
Assume that $\mu = O(\ep^{\zeta})$ for some $\zeta > 0$. For any $\ep \in (0,\ep_0)$,
\begin{equation}\label{eq:exps1}
J_{4\ep}(\dxi) = {2k \over N} A_1 + \mu^{(N-2)p-2} \left[F_{1\ep}(\dxi) + o(1)\right],
\end{equation}
\begin{equation}\label{eq:exps2}
\begin{aligned}
J_{5\ep}(\dxi) &= - \ep \log \mu \left[\frac{\alpha}{(p+1)^2} + \frac{\beta}{(q+1)^2}\right] k NA_1 \\
&+ \ep k \left[ \frac{\alpha A_7}{p+1} + \frac{\beta A_6}{q+1} -
\left\{\frac{\alpha}{(p+1)^2} + \frac{\beta}{(q+1)^2}\right\} kA_1\right] + \ep F_{3\ep}(\dxi) + o(\ep)
\end{aligned}
\end{equation}
and
\begin{equation}\label{eq:exps3}
J_{6\ep}(\dxi) = o(\mu^{(N-2)p-2}) + o(\ep)
\end{equation}
uniformly in $\Lambda$.
\end{lemma}
\begin{proof}
Since $J_{4\ep} = J_{1\ep}$, estimate \eqref{eq:exps1} immediately follows from \eqref{eq:exp02}. Using the expansion
\[\frac{a^{t+1-b\ep}}{t+1-b\ep} - \frac{a^{t+1}}{t+1} = \left[ \frac{a^{t+1}b}{(t+1)^2} - \frac{a^{t+1}(\log a)b}{t+1} \right] \ep + o(\ep)\]
for $a \ge 0$, $\beta \in \R$ and $t > 0$, one can derive \eqref{eq:exps2}.
Also, \eqref{eq:exps3} can be obtained as in the proof of Lemmas \ref{lemma:exp01} and \ref{lemma:error}.
\end{proof}

\begin{proof}[Completion of the proof of Theorem \ref{thm:sub1}]
By preceding discussion, we have
\begin{equation}\label{eq:fdefs}
\begin{aligned}
J_\ep(\dxi) &= C_0 - C_5 \ep \log \mu + C_6 \ep + C_1 \mu^{(N-2)p-2}d^{(N-2)p-2} \tta(\xi) - C_7 \ep \log d + \text{(h.o.t.)}
\end{aligned}
\end{equation}
where $C_0, C_1, C_5$ and $C_7$ are positive numbers, and $C_6$ is a real number.
If we take $\mu = \ep^{1 \over (N-2)p-2}$, then \eqref{eq:fdefs} reduces to
\begin{align*}
J_\ep(\dxi) &= C_0 - C_5 \ep \log \mu + C_6 \ep + \ep \underbrace{\(C_1 d^{(N-2)p-2} \tta(\xi) - C_7 \log d\)}_{=: \mfj_4(\dxi)} + \text{(h.o.t.)}.
\end{align*}
By picking sufficiently small $\delta_1, \delta_2 > 0$ in the definition \eqref{eq:Lambda} of $\Lambda$,
we can make the function $\mfj_4$ have a strict minimum point in $\text{int}(\Lambda)$.
Thus $J_\ep$ has a minimum point in $\text{int}(\Lambda)$ provided that $\ep > 0$ is small enough.
This and the reduction process complete the proof.
\end{proof}
\begin{proof}[Completion of the proof of Theorem \ref{thm:sub2}]
To conclude the proof, we combine the arguments in Subsection \ref{subsec:proof12} and the proof of the previous theorem.
We omit the details.
\end{proof}

\appendix
\section{Technical computations and proofs} \label{sec:app}
\subsection{Algebraic lemmas} 
We prove elementary algebraic lemmas.
\begin{lemma}\label{lemma:alg}
If $N \ge 3$, $q \ge p > 0$, and \eqref{eq:hyper} holds, then we have that
\[((N-2)p-2)q > N+2 \quad \text{and} \quad pqq^* > p+1\]
where $q^*$ is the number defined in \eqref{eq:pq-star}.
\end{lemma}
\begin{proof}
Set $A = (p+1)(N-2) > N$. Note that
\begin{align*}
((N-2)p-2)q > N+2 &\Leftrightarrow (A-N)(q+1) > A+2 \\
&\Leftrightarrow \frac{A-N}{A+2} > \frac{1}{q+1} = (N-2)\(\frac{A-N}{AN}\) \\
&\Leftrightarrow AN > (N-2)(A+2) \Leftrightarrow A > N-2
\end{align*}
and
\[\frac{pq}{p+1} > \frac{1}{q^*} \Leftrightarrow \frac{pq-1}{p+1} > \frac{1}{N} \Leftrightarrow \frac{2}{N} = 1 - \frac{1}{q+1} - \frac{1}{p+1} > \frac{1}{N(q+1)} \Leftrightarrow q+1 > \frac{1}{2},\]
in which the rightmost inequalities clearly hold.
\end{proof}

\begin{lemma}\label{lemma:alg2}
If $N \ge 8$, $p > 1$, and \eqref{eq:hyper} holds, then $p+2 > q$.
\end{lemma}
\begin{proof}
Since $(p,q)$ is on the critical hyperbola \eqref{eq:hyper}, it suffices to verify that $p+2 \ge q$ for $p=1$, that is, $3 \ge \frac{N+4}{N-4}$ holds.
The latter inequality is reduced to $N \ge 8$.
\end{proof}

\subsection{Proofs of Lemmas \ref{lemma:Green1} and \ref{lemma:Green2}} \label{subsec:appa2}
This subsection is devoted to the proof of Lemmas \ref{lemma:Green1} and \ref{lemma:Green2}, which regard regularity and symmetry of the function $\wth$.

\begin{proof}[Proof of Lemma \ref{lemma:Green1}]
The first claim follows from the proof of Lemma \ref{lemma:wth}.

\medskip
For the second claim, we use the representation formula
\begin{equation}\label{eq:Green10}
\begin{aligned}
\wth(x,y) &= \int_{\Omega} \left[\frac{\ga_N}{|x-z|^{N-2}} \(\frac{\ga_N}{|z-y|^{N-2}}\)^p
- G(x,z)G^p(z,y) \right] dz \\
&\ + \int_{\R^N \setminus \Omega} \frac{\ga_N}{|x-z|^{N-2}} \(\frac{\ga_N}{|z-y|^{N-2}}\)^p dz
\end{aligned}
\end{equation}
for $(x,y) \in \Omega \times \Omega$.

Fix $x, y \in \Omega$ and a number $r > 0$ so small that $B^N(y,r) \subset \Omega$.
Denote the integrand of the first integral on the right-hand side of \eqref{eq:Green10} by $g_x(y,z)$.
Given any $\ell = 1, \cdots, N$, we denote by $\mathbf{e}_{\ell}$ the vector with 1 in the $\ell$-th coordinate and 0's elsewhere.
Let also $\{t_n\}_{n \in \N}$ be a sequence of numbers in $(-r,r)$ tending to 0 as $n \to \infty$, and
\begin{equation}\label{eq:Green11}
g_n(z) := \frac{g_x(y+t_n\mathbf{e}_{\ell},z) - g_x(y,z)}{t_n} \to \frac{\pa g_x}{\pa y_{\ell}}(y,z)
\quad \text{for } z \in \Omega \setminus \{x,y\}.
\end{equation}
Then there exists a constant $C > 0$ depending only on $N$, $p$, $\Omega$, $y$ and $r$ such that
\begin{align*}
|g_n(z)| &\le \frac{1}{t_n} \int_0^{t_n} \left| \frac{\pa g_x}{\pa y_{\ell}}(y+s\mathbf{e}_{\ell},z) \right| ds \\
&\le \frac{1}{t_n} \int_0^{t_n} \frac{C}{|x-z|^{N-2}} \frac{ds}{|z-(y+s\mathbf{e}_{\ell})|^{(N-2)(p-1)+1}} =: h_n(z)
\end{align*}
for $z \in \Omega \setminus \{x,y\}$. The fundamental theorem of calculus leads us to
\[h_n(z) \to \frac{C}{|x-z|^{N-2}} \frac{1}{|z-y|^{(N-2)(p-1)+1}} =: h_{\infty}(z) \quad \text{for } z \in \Omega \setminus \{x,y\}.\]
Furthermore, we infer using Fubini's theorem that
\begin{align*}
\int_{\Omega} h_n(z) dz &= \frac{1}{t_n} \int_0^{t_n} \underbrace{\int_{\Omega} \frac{C}{|x-z|^{N-2}} \frac{dz}{|z-(y+s\mathbf{e}_{\ell})|^{(N-2)(p-1)+1}}}_{=: \psi(s)} ds \\
&\to \psi(0) = \int_{\Omega} h_{\infty}(z) dz
\end{align*}
as $n \to \infty$, where the Calderon-Zygmund estimate with the assumption that $p < \frac{N-1}{N-2}$ guarantees that $\psi$ is a continuous function in $s \in (-r,r)$.
Hence, by the dominated convergence theorem and \eqref{eq:Green11},
\[\frac{\pa}{\pa y_{\ell}} \left[ \int_{\Omega} g_x(y,z) dz \right] = \lim_{n \to \infty} \int_{\Omega} g_n(z) dz
= \int_{\Omega} \frac{\pa g_x}{\pa y_{\ell}}(y,z) dz\]
for $y \in \Omega$. The second integral on the right-hand side of \eqref{eq:Green10} is easier to handle.

We conclude that
\begin{equation}\label{eq:Green22}
\begin{aligned}
\nabla_y \wth(x,y) &= p \int_{\Omega} \left[ \(\frac{\ga_N}{|x-z|^{N-2}} - H(x,z)\) \(\frac{\ga_N}{|y-z|^{N-2}}-H(z,y)\)^{p-1} \right. \\
&\hspace{90pt} \times \left\{ (N-2)\ga_N \frac{y-z}{|y-z|^N} + \nabla_y H(z,y) \right\} \\
&\left. \hspace{90pt} - (N-2) \ga_N^{p+1} \frac{1}{|x-z|^{N-2}} \frac{y-z}{|y-z|^{(N-2)p+2}} \right] dz \\
&\ - (N-2)p \ga_N^{p+1} \int_{\R^N \setminus \Omega} \frac{1}{|x-z|^{N-2}} \frac{y-z}{|y-z|^{(N-2)p+2}} dz.
\end{aligned}
\end{equation}
One more application of the dominated convergence theorem on \eqref{eq:Green22} shows that the map $y \in \Omega \mapsto \nabla_y \wth(x,y)$ is continuous.
\end{proof}
\begin{proof}[Proof of Lemma \ref{lemma:Green2}]
Differentiating $\wth(x,y)$ in \eqref{eq:Green10} with respect to $x$ and taking $y = \xi$, we obtain
\begin{align}
\nabla_x \wth(x,\xi) &= (N-2)\ga_N \int_{\Omega} \frac{x-z}{|x-z|^N} \left[ \(\frac{\ga_N}{|\xi-z|^{N-2}} - H(z,\xi)\)^p - \frac{\ga_N^p}{|\xi-z|^{(N-2)p}} \right] dz \nonumber \\
&\ + \int_{\Omega} \nabla_xH(x,z) \(\frac{\ga_N}{|\xi-z|^{N-2}} - H(z,\xi)\)^p dz \label{eq:Green21} \\
&\ - (N-2)\ga_N^{p+1} \int_{\R^N \setminus \Omega} \frac{x-z}{|x-z|^N} \frac{1}{|\xi-z|^{(N-2)p}} dz \nonumber
\end{align}
provided that $p \in (\frac{2}{N-2}, \frac{N-1}{N-2})$.

We insert $x = y = \xi$ into \eqref{eq:Green22} and $x = \xi$ into \eqref{eq:Green21}, respectively, and then compare the results applying the identity that $H(\xi,z) = H(z,\xi)$ for all $z,\, \xi \in \Omega$.
Then we see that \eqref{eq:Green20} is true.
\end{proof}

\subsection{Proofs of Corollaries \ref{cor:dec1} and \ref{cor:dec2}} \label{subsec:appa3}
This subsection is devoted to the proof of Corollaries \ref{cor:dec1} and \ref{cor:dec2}, which concern the decay of the pair $(U_{1,0}(x), V_{1,0}(x))$ as $|x| \to \infty$.
\begin{proof}[Proof of Corollary \ref{cor:dec1}]
We first derive \eqref{eq:dec3}. By \eqref{eq:dec1}, the Kelvin transform $V_{1,0}^*$ of $V_{1,0}$ satisfies that
\[V_{1,0}^*(0) = \lim_{|x| \to 0} V_{1,0}^*(x) = \lim_{|x| \to \infty} |x|^{N-2} V_{1,0}(x) = b_{N,p}\]
and
\[-\Delta V_{1,0}^*(x) = \frac{1}{|x|^{N+2}} (-\Delta V_{1,0})\(\frac{x}{|x|^2}\)
= \frac{1}{|x|^{N+2}} U_{1,0}^q\(\frac{x}{|x|^2}\) = O\(|x|^{((N-2)p-2)q-(N+2)}\)\]
for $x \in \R^N$ near $0$. Observe that the term $o(1)$ tends to $0$ as $|x| \to 0$ uniformly, because $U$ is radial.
Besides, as confirmed in Lemma \ref{lemma:alg}, it is true that $(N-2)p-2)q > N+2$.
Thus we find from elliptic regularity that $V_{1,0}^* \in C^{1,\sigma}(B^N(0,2))$ for any $\sigma \in (0,1)$. In particular,
\[|V_{1,0}^*(x) - b_{N,p}| = |V_{1,0}^*(x) - V_{1,0}^*(0)| \le C|x| \quad \text{in } B^N(0,1).\]
Writing the inequality in terms of $V_{1,0}$, we obtain \eqref{eq:dec3}.

We turn to the proof of \eqref{eq:dec4}. By \eqref{eq:dec7},
\[-\Delta (\pa_{\ell} V_{1,0})^*(x) = \frac{q}{|x|^{N+2}} U_{1,0}^{q-1}\(\frac{x}{|x|^2}\) \pa_{\ell} U_{1,0}\(\frac{x}{|x|^2}\) = O\(|x|^{((N-2)p-2)q-(N+1)}\)\]
for $\ell = 1, \cdots, N$ and $x \in \R^N$ near $0$, and the term $o(1)$ again tends to $0$ as $|x| \to 0$ uniformly.
Elliptic regularity then yields that $(\pa_{\ell} V_{1,0})^* \in C^{1,\sigma}(B^N(0,2))$ for any $\sigma \in (0,1)$, and so
\[|(\pa_{\ell} V_{1,0})^*(x) - (\pa_{\ell} V_{1,0})^*(0) - \nabla(\pa_{\ell} V_{1,0})^*(0) \cdot x| \le C|x|^{2-\zeta} \quad \text{in } B^N(0,1)\]
for a fixed number $\zeta \in (0,1)$.

By the representation formula, we have
\begin{align*}
(\pa_{\ell} V_{1,0})^*(x) &= q \int_{\R^N} \frac{\ga_N}{|x-y|^{N-2}}
\cdot U_{1,0}^{q-1}\(\frac{y}{|y|^2}\) \pa_{\ell} U_{1,0}\(\frac{y}{|y|^2}\) \frac{dy}{|y|^{N+2}} \\
&= q \int_{\R^N} \frac{\ga_N}{|x-z|z|^{-2}|^{N-2}}
\cdot U_{1,0}^{q-1}(z)\, \pa_{\ell} U_{1,0}(z) \frac{dz}{|z|^{N-2}} \quad \(\text{substitute } z = \frac{y}{|y|^2}\)
\end{align*}
for $x \in \R^N$ near 0. Taking $x = 0$ above and using the radial symmetry of $U_{1,0}$, we immediately obtain that $(\pa_{\ell} V_{1,0})^*(0) = 0$. Also, for $m = 1, \cdots, N$,
\begin{align*}
\pa_m (\pa_{\ell} V_{1,0})^*(0) = (N-2)q \ga_N \int_{\R^N} z_m U_{1,0}^{q-1}(z)\, \pa_{\ell} U_{1,0}(z)\, dz.
\end{align*}
Clearly, its right-hand side is nonzero only if $m = \ell$. In this case, by integrating by parts, and employing the representation formula of $V_{1,0}^*$ and \eqref{eq:dec1}, we derive
\[q \int_{\R^N} z_{\ell}\, U_{1,0}^{q-1}(z)\, \pa_{\ell} U_{1,0}(z)\, dz = \int_{\R^N} z_{\ell}\, \pa_{\ell} U_{1,0}^q(z)\, dz = - \int_{\R^N} U_{1,0}^q(z)\, dz = -\ga_N^{-1} b_{N,p}.\]
As a consequence,
\[\left| (\pa_{\ell} V_{1,0})^*(x) + (N-2)b_{N,p}\,x_{\ell} \right| \le C|x|^{2-\zeta}\]
for $x \in \R^N$ near 0, which is an equivalent form to \eqref{eq:dec4}.
\end{proof}
\begin{proof}[Proof of Corollary \ref{cor:dec2}]
Select any $\kappa > 2$. Reasoning as in the completion of the proof of Proposition 4.1 in \cite{CK} (see also Appendix \ref{subsec:ineq} below), one can check that
\begin{equation}\label{eq:intest}
\int_{\R^N \setminus B^N(0,1)} \frac{1}{|x-y|^{N-2}}\frac{dy}{|y|^{\kappa}} \le \frac{C}{|x|^{\min\{N-2, \kappa-2\}}} \quad \text{for } |x| \ge 2.
\end{equation}
By means of \eqref{eq:dec1}, \eqref{eq:dec3} and \eqref{eq:intest}, we compute
\begin{align*}
&\ U_{1,0}(x)-\frac{a_{N,p}}{|x|^{(N-2)p-2}} \\
&= \int_{\R^N} \frac{\ga_N}{|x-y|^{N-2}} \(V_{1,0}^p(y) - \frac{b_{N,p}^p}{|y|^{(N-2)p}}\) dy \\
&= O\(\int_{B^N(0,1)} \frac{1}{|x-y|^{N-2}}\frac{dy}{|y|^{(N-2)p}}\) + O\(\int_{\R^N \setminus B^N(0,1)} \frac{1}{|x-y|^{N-2}}\frac{dy}{|y|^{\min\{(N-2)p+1, (N-1)p\}}}\) \\
&= O\(\frac{1}{|x|^{N-2}}\) + O\(\frac{1}{|x|^{\min\{(N-2)p-1, (N-1)p-2\}}}\) = O\(\frac{1}{|x|^{(N-2)p-1}}\)
\end{align*}
for $|x| \ge 2$ and for $p \in [1, \frac{N-1}{N-2})$. Therefore \eqref{eq:dec5} is true.

Using \eqref{eq:dec4} and arguing as before, we also obtain \eqref{eq:dec6}.
\end{proof}

\subsection{Derivation of \eqref{eq:ineq1} and \eqref{eq:ineq2}} \label{subsec:ineq}
In this subsection, we derive two inequalities needed in the proof of Lemma \ref{lemma:wtg}.

\medskip
We first prove inequality \eqref{eq:ineq1}, by considering three mutually exclusive cases.

\medskip \noindent \textsc{Case 1}: Assume that $|x-\xi_l| \le M\mu_l$ for some large $M > 1$.
Setting $x_0 = \mu_l^{-1}(x-\xi_l)$, we find that $|x_0| \le M$ and
\begin{multline}\label{eq:ineq11}
\mu_l^{N \over q+1} \int_{B^N(\xi_l, \mu_l^{\kappa_1})} \frac{1}{|x-y|^{N-2}} V_l^{p-1}(y) dy \\
= \mu_l^{{N \over q+1}-{(p-1)N \over p+1}+2} \int_{B^N(0,\mu_l^{\kappa_1-1})} \frac{1}{|y-x_0|^{N-2}} V_{1,0}^{p-1}(y) dy.
\end{multline}
Then we estimate the integral on the right-hand side of \eqref{eq:ineq11} by decomposing the domain of integration into
\[B^N(0,\mu_l^{\kappa_1-1}) = B^N(x_0,1) \cup \(B^N(0,2M) \setminus B^N(x_0,1)\) \cup \(B^N(0,\mu_l^{\kappa_1-1}) \setminus B^N(0,2M)\),\]
getting that it is bounded by $\mu^{(\kappa_1-1)(N-(N-2)p)}$. Hence
\begin{equation}\label{eq:ineq12}
\mu^{N \over q+1} \int_{B^N(\xi_l, \mu_l^{\kappa_1})} \frac{1}{|x-y|^{N-2}} V_l^{p-1}(y) dy \le C\mu^{Np \over q+1} \mu^{(N-(N-2)p)\kappa_1}.
\end{equation}

\medskip \noindent \textsc{Case 2}: Assume that $M\mu_l < |x-\xi_l| < 2\mu_l^{\kappa_1}$.
In this case, we have that $M < |x_0| < 2\mu_l^{\kappa_1-1}$ and \eqref{eq:ineq11} holds.
Once more, we estimate the integral on the right-hand side of \eqref{eq:ineq11} by decomposing the domain of integration into
\[B^N(0,\mu_l^{\kappa_1-1}) = B^N(0,r_0) \cup B^N(x_0,r_0) \cup \left[B^N(0,\mu_l^{\kappa_1-1}) \setminus \(B^N(0,r_0) \cup B^N(x_0,r_0)\)\right]\]
where $r_0 = \frac{|x_0|}{2}$, getting that it is bounded by $\mu^{(\kappa_1-1)(N-(N-2)p)}$. Thus, \eqref{eq:ineq12} is true.

\medskip \noindent \textsc{Case 3}: Assume that $2\mu_l^{\kappa_1} \le |x-\xi_l| \le C$. Then
\[|x-y| \ge |x-\xi_l| - |y-\xi_l| \ge \mu_l^{\kappa_1} \quad \text{for } y \in B^N(\xi_l, \mu_l^{\kappa_1}).\]
Therefore,
\begin{align*}
\mu^{N \over q+1} \int_{B^N(\xi_l, \mu_l^{\kappa_1})} \frac{1}{|x-y|^{N-2}} V_l^{p-1}(y) dy
&\le C\mu^{{N \over q+1}-{N(p-1) \over p+1}} \mu^{N-(N-2)\kappa_1} \int_{B^N(0,\mu_l^{\kappa_1-1})} V_{1,0}^{p-1}(y) dy \\
&\le C\mu^{Np \over q+1} \mu^{(N-(N-2)p)\kappa_1}.
\end{align*}

\medskip
Consequently, \eqref{eq:ineq1} holds. The proof of \eqref{eq:ineq2} is similar.

\section{HLS inequality and elliptic regularity} \label{sec:app-b}
The classical HLS inequality states that for $N \ge 3$ and $r, s > 1$ with $\frac{1}{r} + \frac{1}{s} = \frac{N+2}{N}$,
there exists a constant $C > 0$ depending only on $N$ and $r$ such that
\[\left|\int_{\R^N} \int_{\R^N} \frac{f(x)g(y)}{|x-y|^{N-2}} dx dy \right| \le C \|f\|_{L^r(\R^N)}\|g\|_{L^s(\R^N)}\]
for $f \in L^r(\R^N)$ and $g \in L^s(\R^N)$.
The dual statement is that for any $t > \frac{N}{N-2}$, there exists $C > 0$ depending only on $N$ and $t$ such that
\begin{equation}\label{eq:HLS}
\left\||\cdot|^{2-N} \ast h \right\|_{L^t(\R^N)} \le C \|h\|_{L^{Nt \over N+2t}(\R^N)}
\end{equation}
for $h \in L^{Nt \over N+2t}(\R^N)$.

In this section, we will deduce an elliptic regularity result based on the HLS inequality \eqref{eq:HLS}
and apply it to derive the uniform boundedness of $(\Psi_\dxi^\ep, \Phi_\dxi^\ep)$ on $\Omega$, as claimed in Proposition \ref{prop:nonlin}.
Refer to \cite{CL,CJLL,Ha,CK} and references therein for related works.
\begin{lemma}\label{lemma:reg}
Suppose that $N \ge 3$ and $q \ge p > 1$ satisfy \eqref{eq:hyper}.
Assume also that $(\Psi, \Phi) \in X$, $(P_1, P_2) \in (L^{\infty}(\Omega))^2$, and $(Q_1, Q_2) \in (L^{\sigma}(\Omega))^2$ with $\sigma > \frac{N}{2}$.
Finally, pick any $r > \frac{q+1}{2}$ and $s > \frac{N}{N-2}$ such that
\begin{equation}\label{eq:sr}
\frac{1}{s} = \frac{N+2r}{Nr} - \frac{p-1}{p+1} = \frac{1}{r} + \frac{1}{p+1} - \frac{1}{q+1} \ge \frac{1}{r}.
\end{equation}
There exists a small constant $\delta > 0$ depending only on $N, p$ and $\Omega$ such that if
\[\|F_1\|_{L^{p+1 \over p-1}(\Omega)} + \|F_2\|_{L^{q+1 \over q-1}(\Omega)} < \delta\]
and
\begin{equation}\label{eq:reg1}
\begin{cases}
-\Delta \Psi = F_1 (\Phi + P_1) + Q_1 &\text{in } \Omega,\\
-\Delta \Phi = F_2 (\Psi + P_2) + Q_2 &\text{in } \Omega,\\
\Psi = \Phi = 0 &\text{on } \pa \Omega,
\end{cases}
\end{equation}
then $(\Psi, \Phi) \in L^r(\Omega) \times L^s(\Omega)$.
\end{lemma}
\begin{proof}
The proof closely follows the lines of the proof of Theorem 1.3 in \cite{CL}.

Fix any pair $(r,s)$ described in the statement.
Let $T_1$ and $T_2$ be the operators given as
\[(T_1g)(x) = \int_{\Omega} G(x,y) F_1(y)(g+P_1)(y) dy
\quad \text{and} \quad
(T_2f)(x) = \int_{\Omega} G(x,y) F_2(y)(f+P_2)(y)dy\]
for $x \in \Omega$, where $G$ is the Green's function of the Dirichlet Laplacian in $\Omega$.
By using \eqref{eq:Green}, the HLS inequality \eqref{eq:HLS} and H\"older's inequality, we obtain
\[\|T_1g\|_{L^r(\Omega)} \le C\|F_1(g+P_1)\|_{L^{Nr \over N+2r}(\Omega)} \le C\|F_1\|_{L^{p+1 \over p-1}(\Omega)}\(\|g\|_{L^s(\Omega)} + \|P_1\|_{L^s(\Omega)}\),\]
and similarly,
\[\|T_2f\|_{L^s(\Omega)} \le C\|F_2\|_{L^{q+1 \over q-1}(\Omega)}\(\|f\|_{L^r(\Omega)} + \|P_2\|_{L^r(\Omega)}\).\]
Therefore, if we define the operator $T$ by $T(f,g) = (T_1g, T_2f)$, then it maps $L^r(\Omega) \times L^s(\Omega)$ into itself.
In fact, \eqref{eq:A1A2} indicates that $T$ is a contraction mapping on $L^r(\Omega) \times L^s(\Omega)$.

Set the functions
\[q_1 = \int_{\Omega} G(\cdot,y) Q_1(y) dy
\quad \text{and} \quad
q_2 = \int_{\Omega} G(\cdot,y) Q_2(y) dy\]
which belong to $L^{\infty}(\Omega)$ thanks to the condition $(Q_1, Q_2) \in (L^{\sigma}(\Omega))^2$ with $\sigma > \frac{N}{2}$.
We also write \eqref{eq:reg1} in the operator form
\begin{equation}\label{eq:reg3}
(\Psi, \Phi) = T(\Psi, \Phi) + (q_1, q_2).
\end{equation}
Then, by invoking the contraction mapping theorem and the uniqueness of solutions to \eqref{eq:reg3},
we deduce that $(\Psi, \Phi) \in L^r(\Omega) \times L^s(\Omega)$ for all pairs $(r,s)$ depicted in the statement (refer to Theorem 1 in \cite{CJLL}).
The proof is finished.
\end{proof}

\begin{proof}[Proof of the boundedness part in Proposition \ref{prop:nonlin}]
For fixed $\ep \in (0,\ep_0)$ and $(\dxi) \in \Lambda$, we denote $(\Psi,\Phi) = (\Psi_\dxi^\ep, \Phi_\dxi^\ep)$ for the sake of brevity. Equation \eqref{eq:aux1} reads
\[\begin{cases}
\displaystyle -\Delta \Psi = \left|\sum_{i=1}^k PV_i + \Phi\right|^{p-1} \(\sum_{i=1}^k PV_i + \Phi\) + \wtq_1 &\text{in } \Omega,\\
\displaystyle -\Delta \Phi = \left|\mbP U_\dxi + \Psi\right|^{q-1} \(\mbP U_\dxi + \Psi\) + \wtq_2 &\text{in } \Omega,\\
\Psi = \Phi = 0 &\text{on } \pa \Omega
\end{cases} \]
where $\wtq_1, \wtq_2 \in C^{\infty}(\ovom)$. Hence
\begin{equation}\label{eq:reg2}
\begin{cases}
-\Delta \Psi = |\Phi|^{p-1} (\Phi+P_1) + Q_1 &\text{in } \Omega,\\
-\Delta \Phi = |\Psi|^{q-1} (\Psi+P_2) + Q_2 &\text{in } \Omega,\\
\Psi = \Phi = 0 &\text{on } \pa \Omega
\end{cases}
\end{equation}
where we take $P_1 = p\, (\sum_{i=1}^k PV_i)^{p-1}$, $P_2 = \mbP U_\dxi \in C^{\infty}(\ovom)$ and some $Q_1, Q_2 \in L^{\infty}(\Omega)$.

By virtue of \eqref{eq:nonlin}, all the hypotheses in Lemma \ref{lemma:reg} are fulfilled for system \eqref{eq:reg2},
and so $(\Psi,\Phi) \in L^r(\Omega) \times L^s(\Omega)$ for $r > \frac{q+1}{2}$ arbitrarily large and $s$ satisfying \eqref{eq:sr}.
This with elliptic regularity applied to \eqref{eq:reg2} imply that $(\Psi,\Phi) \in (L^{\infty}(\Omega))^2$.
\end{proof}

\end{document}